\documentclass[10pt,leqno]{article}
\usepackage{amsmath,amsfonts,amscd,amsthm}
\newcommand{\dps}{\displaystyle}
\newcommand{\bu}{{\bf u}}
\newcommand{\uhb}{\bu_{h,{\beta}}}
\newcommand{\hbu}{\hat{\bf u}}

\newcommand{\td}{{\tilde\Delta}}
\newcommand{\pr}{\partial}
\newcommand{\pt}{\partial_t}
\newcommand{\bw}{{\bf w}}
\newcommand{\e}{{\bf e}}

\newcommand{\bJ}{{\bf J}}

\newcommand{\bU}{{\bf U}}
\newcommand{\tbU}{\tilde{\bU}}

\newcommand{\bv}{{\bf v}}

\newcommand{\bta}{\mbox{\boldmath $\eta$}}
\newcommand{\hbta}{\hat{\bta}}
\newcommand{\tbta}{\tilde{\bta}}

\newcommand{\bphi}{\mbox{\boldmath $\phi$}}

\newcommand{\bxi}{\mbox{\boldmath $\xi$}}
\newcommand{\tbxi}{\tilde{\bxi}}
\newcommand{\hbxi}{\hat{\bxi}}
\newcommand{\bV}{{\bf V}}
\newcommand{\bW}{{\bf W}}
\newcommand{\bH}{{\bf H}}
\newcommand{\bL}{{\bf L}}
\newcommand{\f}{{\bf f}}
\newcommand{\tf}{\tilde{\f}}
\newcommand{\g}{{\bf g}}

\newcommand{\ve}{\varepsilon}

\begin{document}

\newcommand{\se}{\setcounter{equation}{0}}
\def\theequation{\thesection.\arabic{equation}}

\newtheorem{theorem}{Theorem}[section]
\newtheorem{cdf}{Corollary}[section]
\newtheorem{lemma}{Lemma}[section]
\newtheorem{remark}{Remark}[section]

\title
{ \large\bf Backward Euler method for the Equations of Motion \\
             Arising in Oldroyd Fluids of Order One \\
            with Nonsmooth Initial Data }

\author{Deepjyoti Goswami {\footnote{
Department of Mathematics, Federal University of Paran\'a,
Centro Polit\'ecnico, Curitiba,
Cx.P: 19081, CEP: 81531-990,PR, Brazil.}}
and 
Amiya K. Pani {\footnote{
Department of Mathematics, Industrial Mathematics Group,
Indian Institute of Technology Bombay, Powai, Mumbai-400076, India.}}}

\maketitle

\begin{abstract}
In this paper, a backward Euler method is discussed for the equations of motion arising in 
the 2D Oldroyd model of viscoelastic fluids of order one with the forcing term independent of 
time or in $L^{\infty}$ in time. It is shown that the estimates of the discrete solution in Dirichlet norm is bounded uniformly in time.
Optimal {\it a priori} error estimate in $\bL^2$-norm is derived for the discrete problem with non-smooth initial data. This estimate is shown to be
uniform in time, under the assumption of uniqueness condition.
\end{abstract}

\vspace{1em} 

\section{Introduction}
In this paper, we consider fully-discrete approximations to the equations of motion arising in the Oldroyd fluids (see Oldroyd\cite{Old}, Oskolkov\cite{Os89}) of order one:
\begin{eqnarray}\label{om}
~~\frac {\partial \bu}{\partial t}+\bu\cdot\nabla\bu-\mu\Delta\bu-\int_0^t \beta
 (t-\tau)\Delta\bu (\tau)\,d\tau+\nabla p=\f, ~~~~~\mbox {in}~ \Omega,~t>0
\end{eqnarray}
with incompressibility condition
\begin{eqnarray}\label{ic}
 \nabla \cdot \bu=0,~~~~~\mbox {on}~\Omega,~t>0,
\end{eqnarray}
and initial and boundary conditions
\begin{eqnarray}\label{ibc}
 \bu(x,0)=\bu_0~~\mbox {in}~\Omega,~~~\bu=0,~~\mbox {on}~\partial\Omega,~t\ge 0.
\end{eqnarray}
Here, $\Omega$ is a bounded domain in $\mathbb{R}^2$ with boundary $\partial \Omega$,
$\mu = 2 \kappa\lambda^{-1}>0$ and the kernel $\beta (t) = \gamma \exp (-\delta t),$
where $\gamma= 2\lambda^{-1}(\nu-\kappa \lambda^{- 1})>0$ and $\delta 
=\lambda^{-1}>0$. Further, $\f$ and $\bu_0$ are given functions in their
respective domain of definition. For more details, we refer to \cite{AS89} and \cite{Old}.

There is considerable amount of literature devoted to Oldroyd model by Oskolkov, 
Kotsiolis, Karzeeva, Sobolevskii etc, see \cite{AS89,EO92,KrKO91,KOS92,Os89} 
and recently by Lin {\it et al.} \cite{HLST,HLSST,WHL}, Pani {\it et al.} \cite{PY05, 
PYD06}, Wang {\it et al.} \cite{WHS}, and references, therein. A detailed report on the continuous and semi-discrete cases 
can be found in \cite{GP11}.

Literature for the fully-discrete approximations to the problem (\ref{om})-(\ref{ibc}) is, 
however, limited. In \cite{AO}, Akhmatov and Oskolkov have discussed stable and convergent
 finite difference schemes for the problem (\ref{om})-(\ref{ibc}). Recently in \cite{PYD06}, a 
linearized backward Euler method is used to discretize in temporal direction and semi-group 
theoretic approach is then employed to establish {\it a priori} error estimates.  The following 
error bounds are proved in \cite{PYD06} for $t_n>0$
$$ \|\bu(t_n)-\bU^n\| \le Ce^{-\alpha t_n}k $$
and
$$ \|\bu(t_n)-\bU^n\|_1 \le Ce^{-\alpha t_n}k(t_n^{-1/2}+\log\frac{1}{k}) $$
for smooth initial data and for zero forcing term. Here, $k$ is the time step and $\bU^n$ is 
the finite difference approximation to $\bu(t_n),$ when modified backward Euler method is 
applied in the temporal direction.  Recently Wang {\it et al.} \cite{WHS} have again applied 
backward Euler method for the problem (\ref{om})-(\ref{ibc}), with smooth initial data, when 
the forcing function is non-zero. They have used energy arguments along with uniqueness 
condition to obtain the following uniform error estimates:
$$ \|\bu(t_n)-\bU^n\| \le C(h^2+k) $$
and
$$ \tau^{1/2}\|\bu(t_n)-\bU^n\|_1 \le C(h+k), $$
where $\tau(t_n)= \min\{1,t_n\}$ and $h$ is the mesh size, again with smooth initial data.

Our present investigation is a continuation of \cite{GP11}, where {\it a priori} estimates and 
regularity results have been established, which are uniform in time under realistically 
assumed regularity on the exact solution and when $\f,\f_t\in L^\infty(\bL^2)$. Error 
estimates for semi-discrete Galerkin approximations have been shown to be optimal in 
$L^{\infty}(\bL^2)$-norm for non-smooth initial data. Further, uniform (in time) error 
estimates under uniqueness condition are also established.

In the present article, we discuss backward Euler method to discretize in the temporal variable and 
Galerkin approximations to discretize spatial variables for approximating solutions of the 
problem (\ref{om})-(\ref{ibc}). Our main aim, in this work, is to present optimal error estimate 
for the backward Euler method, when the initial data is non-smooth, that is, 
$\bu_0\in\bJ_1.$  The main results of this paper are follows:

\begin{itemize}
\item [(i)] Proving uniform bound in time in the Dirichlet  norm for the solution of the completely discrete backward Euler method.
\item [(ii)]  Deriving new estimates  which are valid uniformly in time for the error
associated with discrete linearized problem 
\item [(iii)] Establishing estimates for the error related to nonlinear part in which the error constant
depends exponentially in time and thereby, making final error estimate in the velocity to depend
on  exponentially in time.
\item [(iv)] Proving optimal error estimates for the velocity in $\bL^2$-norm which are uniform in time
under the uniqueness assumption.
\end{itemize}

To prove estimate in the Dirichlet norm for the discrete solution which is valid for all time, the usual tool, in the case of the Navier-Stokes equations, is to apply discrete version of uniform Gronwall's Lemma. 
Now for proving $(i),$ it is difficult to apply uniform Gronwall's Lemma due to presence of the discrete version of integral term. Therefore, a new way of looking at the proof has helped to achieve $(i),$ see; Lemma 4.3.  
For $(ii)-(iii),$ we observe that there are difficulties due to the non-linear term along 
with the presence of integral term in the case of non-smooth initial data. For example, the preliminary result ($L^{\infty}(\bL^2)$ estimate) is sub-optimal due to non-smooth initial data (see; Lemma \ref{pree}). 
In order to compensate the loss in the  order of convergence, a more
appropriate tool is  to multiply by $t.$  But, unfortunately, it fails here due to the presence of the integral term (or the summation term). To overcome this difficulty, we modify some tools from the error analysis of linear parabolic integro-differential equations  with non-smooth data  (see; \cite{PS98, PS198, TZ89}) to fit into the present problem and also a special care is taken to bound the nonlinear term. Our analysis makes use of the 
solution,say; $\bV^n$ of a linearized  discrete problem (see; (5.5)) as an  intermediate solution. Then, with its help,
we split the error: $\bu_h^n-\bU^n$ at time level $t=t_n,$ where $\bu_h^n=\bu_h(t_n)$ is the solution of the semi-discrete scheme at $t=t_n$ and $\bU^n$ is the solution of the backward Euler method, into two error components: one in $\bxi^n:=\bu_h^n-\bV^n,$ which denotes the contribution due to the linearized part (see; (\ref{errsplit})), and the other in $\bta^n:=\bU^n-\bV^n,$ which is due to the non-linearity  (see; (\ref{errsplit})). Using a backward discrete linear problem and duality type argument along with an estimate of $\hat{\bxi^n},$ where
$$\hat{\bxi^n}:= k \sum_{j=0}^n \bxi^j,$$
an $L^2$-estimate of $\bxi^n$ which is valid for all time is derived, refer to Theorm 5.1. For $L^2$ estimate of $\bta^n,$ we employ negative norm estimate and $L^2$ estimate of $\hat{\bta}^n$ and obtain estimate which depends on exponentially in time, see; Lemma 5.9. Thus, one of the main result for nonsmooth initial data that we have derived in Theorem 5.2 is as follows: 
\begin{equation} \label{main-velocity-estimate} 
\|\bu(t_n)-\bU^n\| \le K_T t_n^{-1/2}\big(h^2+k(1+\log \frac{1}{k})^{1/2}\big),
\end{equation}
where $K_T$ depends exponentially on time. 
Finally for the proof of $(iv),$ a careful use of the uniqueness condition, it is also shown that the error estimate (\ref{main-velocity-estimate}) is valid for all time.

The remaining part of this paper is organized as follows. In Section $2$, we discuss some notations, basic assumptions and  weak formulations. In Section $3$, a semidiscrete 
Galerkin method is discussed briefly. Section $4$ is devoted to backward Euler method. 
Optimal and uniform error bounds are obtained for the velocity when the initial 
data are in $\bJ_1.$ Finally, we summarize our results in the Section $5.$

\section{Preliminaries}
\se
For our subsequent use, we denote by bold face letters the $\mathbb{R}^2$-valued
function space such as
\begin{eqnarray*}
 {\bf H}_0^1 = [H_0^1(\Omega)]^2, \;\;\; {\bf L}^2 = [L^2(\Omega)]^2
 ~~\mbox {and }\;\; {\bf H}^m=[H^m(\Omega)]^2,
\end{eqnarray*}
where $H^m(\Omega)$ is the standard Hilbert Sobolev space of order $m$. Note
that $\bH^1_0$ is equipped with a norm
$$ \|\nabla\bv\|= \left(\sum_{i,j=1}^{2}(\partial_j v_i, \partial_j
 v_i)\right)^{1/2}=\left(\sum_{i=1}^{2}(\nabla v_i, \nabla v_i)\right)^{1/2}. $$
Further, we introduce some more function spaces for our future use:---
\begin{eqnarray*}
 {\bf J}_1 &=& \{{\bphi} \in {\bf {H}}_0^1 : \nabla \cdot \bphi = 0\}\\
 {\bf J}= \{\bphi \in {\bf {L}}^2 :\nabla \cdot \bphi &=& 0~~{\mbox {\rm in}}~~
 \Omega,\bphi\cdot{\bf{n}}|_{\pr \Omega}=0~~{\mbox {\rm holds}~~{\rm weakly}}\}, 
\end{eqnarray*}
where ${\bf {n}}$ is the outward normal to the boundary $\pr \Omega$ and $\bphi
\cdot {\bf {n}} |_{\pr \Omega} = 0$ should be understood in the sense of trace
in $\bH^{-1/2}(\partial \Omega)$, see \cite{temam}. Let $H^m/\mathbb{R}$ be the
quotient space consisting of equivalence classes of elements of $H^m$ differing
by constants, which is equipped with norm $\| p\|_{H^m /\mathbb{R}}=\| p+c\|_m$.
For any Banach space $X$, let $L^p(0, T; X)$ denote the space of measurable $X$
-valued functions $\bphi$ on  $ (0,T) $ such that
$$ \int_0^T \|\bphi (t)\|^p_X~dt <\infty~~~{\mbox {\rm if}}~~1 \le p < \infty, $$
and for $p=\infty$
$$ {\dps{ess \sup_{0<t<T}}} \|\bphi (t)\|_X <\infty~~~{\mbox {\rm if} }~~p=\infty. $$
Through out this paper, we make the following assumptions:\\
(${\bf A1}$). For ${\bf {g}} \in \bL^2$, let the unique pair of solutions $\{\bv
\in {\bf{J}}_1, q \in L^2 /R\} $ for the steady state Stokes problem
\begin{eqnarray*}
 -\Delta {\bv} + \nabla q = {\bf {g}},\\
 \nabla \cdot\bv = 0\;\;\; {\mbox {\rm in} }~~~\Omega,~~~~\bv|_{\pr \Omega}=0,
\end{eqnarray*}
satisfy the following regularity result
$$  \| \bv \|_2 + \|q\|_{H^1 /R} \le C\|{\bf {g}}\|. $$
\noindent
(${\bf A2}$). The initial velocity $\bu_0$ and the external force $\f$ satisfy for
positive constant $M_0,$ and for $T$ with $0<T \leq \infty$
$$ \bu_0\in {\bf J}_1,~\f,\f_t \in L^{\infty} (0, T ; {\bf L}^2)~~~{\mbox
 {\rm with}}~~~\|\bu_0\|_1 \le M_0,~~{\dps{\sup_{0<t<T} }}\big\{\|\f\|,
 \|\f_t\|\big\} \le M_0. $$


For our subsequent analysis, we present the following Lemma, which can be seen as a 
discrete version of Lemma 2.2 from \cite{PY05}.
\begin{lemma}
Let $g_i,\phi^i \in \mathbb{R},~1\le i \le n,~n\in \mathbb{N}$ and $0<k<1$. Then the 
following estimate holds
$$  \Big(k\sum_{i=1}^n\big(k\sum_{j=1}^i g_{i-j} \phi^j\big)^2 \Big)^{1/2} \le 
\Big(k\sum_{i=1}^k |g_i|\Big) \Big(k\sum_{i=1}^n |\phi^i|^2\Big)^{1/2}. $$
\end{lemma}

\section{Semidiscrete  Galerkin Approximations}
\se
From now on, we denote $h$ with $0<h<1$ to be a real positive discretization
parameter tending to zero. Let  ${\bf H}_h$ and $L_h$, $0<h<1$ be two family of
finite dimensional subspaces of ${\bf H}_0^1 $ and $L^2$, respectively,
approximating velocity vector and the pressure. Assume that the following
approximation properties are satisfied for the spaces $\bH_h$ and $L_h$: \\
${\bf (B1)}$ For each $\bw \in {\bf {H}}_0^1 \cap {\bf {H}}^2 $ and $ q \in
H^1/R$ there exist approximations $i_h w \in {\bf {H}}_h $ and $ j_h q \in
L_h $ such that
$$
 \|\bw-i_h\bw\|+ h \| \nabla (\bw-i_h \bw)\| \le K_0 h^2 \| \bw\|_2,
 ~~~~\| q - j_h q \|_{L^2 /R} \le K_0 h \| q\|_{H^1 /R}.
$$
Further, suppose that the following inverse hypothesis holds for $\bw_h\in\bH_h$:
\begin{align}\label{inv.hypo}
 \|\nabla \bw_h\| \leq  K_0 h^{-1} \|\bw_h\|.
\end{align}
For defining the Galerkin approximations, set for $\bv, \bw, \bphi \in \bH_0^1$,
$$ a(\bv, \bphi) = (\nabla \bv, \nabla \bphi) $$
and
$$ b(\bv, \bw,\bphi)= \frac{1}{2} (\bv \cdot \nabla \bw , \bphi)
   - \frac{1}{2} (\bv \cdot \nabla \bphi, \bw). $$
Note that the operator $b(\cdot, \cdot, \cdot)$ preserves the antisymmetric property of
the original nonlinear term, that is,
$$ b(\bv_h, \bw_h, \bw_h) = 0 \;\;\; \forall \bv_h, \bw_h \in {\bH}_h. $$
Now,the semidiscrete Galerkin formulation reads as:  Find $\bu_h(t)
\in {\bf H}_h$ and $p_h(t) \in L_h$ such that $ \bu_h(0)= \bu_{0h} $ and for $t>0$
\begin{eqnarray}\label{dwfh}
 (\bu_{ht}, \bphi_h) +\mu a (\bu_h,\bphi_h) &+& b(\bu_h,\bu_h,\bphi_h)+ a(\uhb, 
\bphi_h) -(p_h, \nabla \cdot \bphi_h) =(\f, \bphi_h), \nonumber \\
&&(\nabla \cdot \bu_h, \chi_h) =0,
\end{eqnarray}
for $\bphi_h\in{\bf H}_h,~\chi_h \in L_h$. Here $\bu_{0h} \in {\bf H}_h $ is a suitable 
approximation of $\bu_0\in {\bf J}_1$ and
\begin{equation}\label{uhb}
\uhb(t)=\int_0^t \beta(t-s) \bu_h(s)~ds.
\end{equation}

\noindent In order to consider a discrete space analogous to $\bJ_1$, we
impose the discrete incompressibility condition on $\bH_h$ and call it as
$\bJ_h$. Thus, we define $\bJ_h,$ as
$$ {\bf J}_h = \{ v_h \in {\bf H}_h : (\chi_h,\nabla\cdot v_h)=0
 ~~~\forall \chi_h \in L_h \}. $$
Note that $\bJ_h$ is not a subspace of $\bJ_1$. With $\bJ_h$ as above, we now introduce
an equivalent Galerkin formulation as: Find $\bu_h(t)\in {\bf J}_h $ such that $\bu_h(0) =
\bu_{0h} $ and for $t>0$
\begin{eqnarray}\label{dwfj}
~~~~ (\bu_{ht},\bphi_h) +\mu a (\bu_h,\bphi_h) + a(\uhb, \bphi_h)
= -b( \bu_h, \bu_h, \bphi_h)+(\f,\bphi_h)~~\forall \bphi_h \in {\bf J}_h.
\end{eqnarray}

Since $\bJ_h$ is finite dimensional, the problem (\ref{dwfj}) leads to a system of
nonlinear integro-differential equations. For global existence of a solution pair
of (\ref{dwfj}), we refer to \cite{PY05}. Uniqueness (of $p$) is obtained on the quotient
space $L_h/N_h$, where
$$ N_h=\{q_h\in L_h:(q_h, \nabla\cdot\bphi_h)=0,\forall \bphi_h\in{\bf H}_h\}. $$
The norm on $ L_h/N_h $ is given by
$$
 \| q_h\|_{L^2/N_h} = {\dps{\inf_{\chi_h \in N_h} }} \| q_h + \chi_h\|. $$
For continuous dependence of the discrete pressure $p_h (t) \in L_h/N_h$ on the
discrete velocity $u_h(t) \in {\bf J}_h$, we assume the following discrete
inf-sup (LBB) condition for the finite dimensional spaces $\bH_h$ and $L_h$:\\
\noindent
${\bf (B2')}$  For every $q_h \in L_h$, there exists a non-trivial function
$\bphi_h \in {\bf H}_h$ and a positive constant $K_0,$ independent of $h,$
such that
$$ |(q_h, \nabla\cdot \bphi_h)| \ge K_0 \|\nabla \bphi_h \|\| q_h\|_{L^2/N_h}. $$
Moreover, we also assume that the following approximation property holds true
for ${\bf J}_h $. \\
\noindent
${\bf (B2)}$ For every $\bw \in {\bf J}_1 \cap {\bf H}^2, $ there exists an
approximation $r_h \bw \in {\bf J_h}$ such that
$$ \|\bw-r_h\bw\|+h \| \nabla (\bw - r_h \bw) \| \le K_5 h^2 \|\bw\|_2 . $$
This is a less restrictive condition than (${\bf B2'}$) and it has been used to
derive the following properties of the $L^2$ projection $P_h:\bL^2\mapsto \bJ_h$.
We now state without proof these results. For a proof, see \cite{HR82}. For $\bphi
\in \bJ_h$, note that
\begin{equation}\label{ph1}
 \|\bphi- P_h \bphi\|+ h \|\nabla P_h \bphi\| \leq C h\|\nabla \bphi\|,
\end{equation}
and for $\bphi \in \bJ_1 \cap \bH^2,$
\begin{equation}\label{ph2}
 \|\bphi-P_h\bphi\|+h\|\nabla(\bphi-P_h \bphi)\|\le C h^2\|\td\bphi\|.
\end{equation}
We now define the discrete  operator $\Delta_h: \bH_h \mapsto \bH_h$ through the
bilinear form $a (\cdot, \cdot)$ as
\begin{eqnarray}\label{do}
 a(\bv_h, \bphi_h) = (-\Delta_h\bv_h, \bphi)~~~~\forall \bv_h, \bphi_h\in\bH_h.
\end{eqnarray}
Set the discrete analogue of the Stokes operator $\td =P \Delta $ as
$\td_h = P_h \Delta_h $. Using Sobolev imbedding and Sobolev inequality, it is
easy to prove the following Lemma
\begin{lemma}\label{nonlin}
Suppose conditions (${\bf A1}$), (${\bf B1}$) and (${\bf  B2}$) are satisfied. Then there 
exists a positive constant $K$ such that for $\bv,\bw,\bphi\in\bH_h$, the following holds:
\begin{equation}\label{nonlin1}
 |(\bv\cdot\nabla\bw,\bphi)| \le K \left\{
\begin{array}{l}
 \|\bv\|^{1/2}\|\nabla\bv\|^{1/2}\|\nabla\bw\|^{1/2}\|\Delta_h\bw\|^{1/2}
 \|\bphi\|, \\
 \|\bv\|^{1/2}\|\Delta_h\bv\|^{1/2}\|\nabla\bw\|\|\bphi\|, \\
 \|\bv\|^{1/2}\|\nabla\bv\|^{1/2}\|\nabla\bw\|\|\bphi\|^{1/2}
 \|\nabla\bphi\|^{1/2}, \\
 \|\bv\|\|\nabla\bw\|\|\bphi\|^{1/2}\|\Delta_h\bphi\|^{1/2}, \\
 \|\bv\|\|\nabla\bw\|^{1/2}\|\Delta_h\bw\|^{1/2}\|\bphi\|^{1/2}
 \|\nabla\bphi\|^{1/2}.
\end{array}\right.
\end{equation}
\end{lemma}

\noindent Examples of subspaces $\bH_h$ and $L_h$ satisfying assumptions (${\bf B1}$),
(${\bf B2}'$), and (${\bf B2}$) can be found in \cite{GR, BP, BF}. \\
We present below, a Lemma, that deals with higher order estimates of $\bu_h,$ which will be
 useful in the error analysis of backward Euler method for non-smooth data.
 \begin{lemma}\label{dth2}
 Suppose conditions (${\bf A1}$), (${\bf B1}$), (${\bf B2}$) and (${\bf B4}$) are
 satisfied. Moreover, let $\bu_h(0)\in\bJ_h$ and $\f$ satisfy the assumption (${\bf A3}$).
Then, $\bu_h,$ the solutions of the semidiscrete Oldroyd problem  (\ref{dwfj}) satisfies the 
following {\it a priori} estimates:
\begin{align}
 \tau^*\|\bu_h\|_2^2+(\tau^*)^{r+1}\|\bu_{ht}\|_r^2 & \le K,~~~~r\in \{0,1\}, 
\label{dth11} \\
 e^{-2\alpha t}\int_0^t e^{2\alpha s}(\tau^*)^r(s)\|\bu_{hs}\|^2_r\,ds & \le K, ~~~~r\in 
\{0,1,2\}, \label{dth12} \\
  e^{-2\alpha t} \int_0^t e^{2\alpha s}(\tau^*)^{r+1}(s)\|\bu_{hss}\|_{r-1}^2~ds & \le 
 K,~~~~r\in \{-1,0,1\},  \label{dth13}
\end{align}
 where $(\tau^*)(t)=\min\{1,t\},~\sigma(t) = \tau^*(t) e^{2\alpha t}$ and K depends on
 the given data, but not on time $T$.
 \end{lemma}
 
 \begin{proof}
 The estimates (\ref{dth11})-(\ref{dth12}) can be proved as in the continuous case, see 
\cite{GP11}. For the final estimate, we differentiate (\ref{dwfj}) to find that, for $\bphi_h \in
 \bJ_h,$
 \begin{eqnarray}\label{dwfjt}
  (\bu_{htt},\bphi_h) +\mu a(\bu_{ht},\bphi_h) & + & \beta(0)a(\bu_h,\bphi_h)-\delta
  \int_0^t \beta(t-s) a(\bu_h(s), \bphi_h)~ds \nonumber \\
  & = & -b(\bu_{ht},\bu_h,\bphi_h)-b(\bu_h,\bu_{ht},\bphi_h)+(\f_t,\bphi_h).
 \end{eqnarray}
Taking $\bphi_h=(\tau^*)^2(t)e^{2\alpha t}\bu_{htt}$ in (\ref{dwfjt}), we obtain
\begin{align}\label{dth001}
(\tau^*)^2(t) & e^{2\alpha t}\|\bu_{htt}\|^2+\frac{\mu}{2}\frac{d}{dt}\big( (\tau^*)^2(t) 
e^{2 \alpha t}\|\bu_{ht}\|_1^2\big) \le \big(\alpha (\tau^*)^2(t)+\tau^*(t)\big) 
e^{2\alpha t} \|\bu_{ht}\|^2 \nonumber \\
& +\gamma (\tau^*)^2(t)e^{2\alpha t}\|\bu_h\|_2\|\bu_{htt}\|+\delta (\tau^*)^2(t) 
e^{2\alpha t}\int_0^t \beta(t-s) \|\bu_h(s)\|_2\|\bu_{htt}\|~ds \nonumber \\
&+(\tau^*)^2(t)e^{2\alpha t}\big(|b(\bu_{ht},\bu_h,\bu_{htt})|+ 
|b(\bu_h,\bu_{ht},\bu_{htt})|+\|f_t\|\|\bu_{htt}\|\big)
\end{align}
Use (\ref{nonlin1}) to find that
\begin{align*}
|b(\bu_{ht},\bu_h,\bu_{htt})|+ |b(\bu_h,\bu_{ht},\bu_{htt})| \le \frac{1}{4} 
\|\bu_{htt}\|^2+K\|\bu_{ht}\|_1^2\|\bu_h\|_2^2.
\end{align*}
Now, using (\ref{dth11})-(\ref{dth12}), we can easily deduce from (\ref{dth001}) that
\begin{equation}\label{dth005}
 (\tau^*)^2\|\bu_{ht}\|_1^2+\mu e^{-2\alpha t}\int_0^t 
 (\tau^*)^2(s) e^{2\alpha s}\|\bu_{hss}\|^2~ds \le K.
\end{equation}
We set $\bphi_h=-\tau^*(t) e^{2\alpha t}\td_h^{-1}\bu_{htt}$ in (\ref{dwfjt}). From 
(\ref{nonlin1}), we see that
$$ b(\bu_{ht},\bu_h,\td_h^{-1}\bu_{htt}) \le K\|\bu_{ht}\|^{1/2}\|\bu_{ht}\|_1^{1/2}
   \|\bu_h\|_1\| \bu_{htt}\|_{-1} $$
and therefore
\begin{align*}
 \mu\frac{d}{dt}(\tau^*(t) e^{2\alpha t}\|\bu_{ht}\|^2) +\tau^*(t) e^{2\alpha t} 
\|\bu_{htt}\|_{-1}^2 \le \big(2\alpha \tau^*(t)+1\big) e^{2\alpha t}\|\bu_{ht}\|_1^2 \\
+C(\mu,\gamma) \tau^*(t) e^{2\alpha t} \|\nabla\bu_h\|^2 
+2\|\f_t\|^2 +C(\mu,\delta)(\int_0^t \beta(t-s) e^{\alpha t}\|\td_h\bu_h(s)\|~ds)^2 \\
+C(\mu) \tau^*(t) e^{2\alpha t}\Big(\|\nabla\bu_h\|^2\|\bu_{ht}\|^2 
+\|\nabla\bu_{ht}\|^2(1+ \|\bu_h\|\|\nabla\bu_h\|)\Big).
\end{align*}
Integrate with respect to time and multiply by $e^{-2\alpha t}$ to conclude
\begin{equation}\label{dth006}
 \tau^*(t)\|\bu_{ht}\|^2+\mu e^{-2\alpha t}\int_0^t \tau^*(s) e^{2\alpha s} 
\|\bu_{hss}\|_{-1}^2 ds \le K.
\end{equation}
Finally, we set $\bphi_h= -e^{2\alpha t}\td_h^{2}\bu_{htt}$ in (\ref{dwfjt}) and proceed as 
above to arrive at
\begin{equation}\label{dth007}
\|\bu_{ht}\|_{-1}^2+\mu e^{-2\alpha t}\int_0^t e^{2\alpha s} 
\|\bu_{hss}\|_{-2}^2 ds \le K.
\end{equation}
This completes the rest of the proof.
 \end{proof}
\noindent The following semi-discrete error estimates are proved in \cite{GP11}.
\begin{theorem}\label{errest}
Let $\Omega$ be a convex polygon and let the conditions (${\bf A1}$)-(${\bf A2}$) and (${\bf 
B1}$)-(${\bf B2}$)
be satisfied. Further, let the discrete initial velocity $\bu_{0h}\in \bJ_h$ with
$\bu_{0h}=P_h\bu_0,$ where $\bu_0\in \bJ_1.$ Then,
there exists a positive constant $C$ such that for $0<T<\infty $ with $t\in (0,T]$
$$ \|(\bu-\bu_h)(t)\|+h\|\nabla(\bu-\bu_h)(t)\|\le Ce^{Ct}h^2t^{-1/2}.$$
Moreover, under the assumption of the uniqueness condition, that is,
\begin{equation}\label{uc}
 \frac{N}{\nu^2}\|\f\|_{\infty} < 1~~~\mbox{and}~~N 
=\sup_{\bu, \bv,{\bf w}\in\bH_0^1(\Omega)}\frac{b(\bu,\bv,{\bf 
w})}{\|\nabla\bu\|\|\nabla\bv\| \|\nabla{\bf w}\|},
\end{equation}
where $\nu=\mu+\frac{\gamma}{\delta}$ and $\|\f\|_{\infty} :=\|\f\|_{L^\infty(0, 
\infty; \bL^2(\Omega))}$ then we have the following uniform estimate:
$$ \|(\bu-\bu_h)(t)\| \le Ch^2t^{-1/2}. $$
\end{theorem}

\section{Backward Euler Method}
\se

For time discretization, we state below some notations. Let $k,~0<k<1,$ be the time 
step and let $t_n=nk,~n\ge 0.$ We define for a sequence $\{\bphi^n\}_{n
\ge 0}\subset\bJ_h,$
$$ \pt\bphi^n=\frac{1}{k}(\bphi^n-\bphi^{n-1}). $$
For continuous function $\bv(t),$ we set $\bv_n=\bv(t_n).$ 
Since backward Euler method is of first order in time, we choose the right
rectangle rule to approximate the integral term in (\ref{dwfj}) as:
\begin{equation}\label{rrr}
 q_r^n(\bphi)=k\sum_{j=1}^{n}\beta_{n-j}\phi^j\approx \int_0^{t_n}
 \beta(t_n-s)\bphi(s)~ds
\end{equation}
where $\beta_{n-j}=\beta(t_n-t_j).$ With $w_{nj}= k\beta(t_n-t_j),$ it is observed that the 
the right rectangle rule is positive in the sense that
\begin{equation}\label{rrp}
k\sum_{i=1}^n q_r^i(\phi)\phi^i= k\sum_{i=1}^n k\sum_{j=0}^i 
\omega_{ij}\phi^j \phi^i \ge 0,~~~~\phi=(\phi^0,\cdots,\phi^N)^T.
\end{equation}
For positivity of the rectangle rule with  $\omega_{n0}=0,$ we refer to  McLean and Thom{\'e}e \cite{MT}.
Note that the error incurred due to right rectangle rule in approximating the integral term is
\begin{align}\label{errrr}
\ve_r^n (\phi) & := \int_0^{t_n} \beta(t_n-s)\bphi(s)~ds-k\sum_{j=1}^{n}\beta_{n-j}\phi^j \\
&\le  Kk\sum_{j=1}^{n}\int_{t_{j-1}}^{t_j}\Big|\frac{\partial}{\partial s}(\beta(t_n-s) 
\bphi(s))\Big|~ds. \nonumber
\end{align}
We present here a discrete version of integration by parts. For sequences $\{a_i\}$ and 
$\{b_i\}$ of real numbers, the following summation by parts holds
\begin{equation}\label{sumbp}
k\sum_{j=1}^i a_jb_j= a_i\hat{b}_i-k\sum_{j=1}^{i-1} (\pt a_{j+1})\hat{b}_j,
\end{equation}
where $\hat{b}_i:=k\sum_{j=1}^i b_j.$ \\
We describe below the backward Euler scheme for the semidiscrete Oldroyd problem 
(\ref{dwfh}): Find $\{\bU^n\}_{n\ge 0}\in\bH_h$ and $\{P^n\}_{n\ge 1}\in L_h$ as 
solutions of the recursive nonlinear algebraic equations ($n\ge 1$) :
\begin{equation}\left.\begin{array}{rcl}\label{fdbeh}
 (\pt\bU^n,\bphi_h)+\mu a(\bU^n,\bphi_h) &+& a(q_r^n(\bU),\bphi_h)
 = (P^n,\nabla\cdot \bphi_h) \\
 &+& (\f^n,\bphi_h)-b(\bU^n,\bU^n,\bphi_h)~~~\forall\bphi_h\in\bH_h, \\
 (\nabla\cdot\bU^n,\chi_h)&=& 0~~~\forall \chi_h \in L_h,~~~n\ge 0.
\end{array}\right\}
\end{equation}
We choose $\bU^0=\bu_{0h}=P_h\bu_0.$ Now, for $\bphi_h\in\bJ_h,$ we seek 
$\{\bU^n\}_{n\ge 0}\in\bJ_h$ such that
\begin{equation}\label{fdbej}
 (\pt\bU^n,\bphi_h)+\mu a(\bU^n,\bphi_h)+a(q_r^n(\bU),\bphi_h)= (\f^n,\bphi_h)
 -b(\bU^n,\bU^n,\bphi_h)~~~\forall\bphi_h\in\bJ_h.
\end{equation}
Using variant of Brouwer fixed point theorem and standard uniqueness arguments, it is easy to show that the discrete
problem (\ref{fdbej}) is well-posed. For a proof, we refer to \cite{G11}. Below
we prove {\it a priori} bounds for the discrete solutions $\{\bU^n\}_{n>0}.$
\begin{lemma}\label{stb}
 Let $0<\alpha<\min\{\delta,\frac{\mu\lambda_1}{2}\}$ and $k_0>0$ be such that for
$0<k<k_0$
$$ 1+\big(\frac{\mu\lambda_1}{2}\big)k\ge e^{\alpha k}. $$
Further, let $\bU^0=\bu_{0h}=P_h\bu_0$ with $\bu_0\in\bJ_1.$ Then, the discrete 
solution $\bU^N,~N\ge 1$ of (\ref{fdbej}) satisfies the following estimates:
\begin{align}\label{stb1a}
\|\bU^N\|^2+\Gamma_1 e^{-\alpha t_N}k\sum_{n=1}^N e^{\alpha t_n}\|\nabla\bU^n\|^2 \le C
\Big(e^{-\alpha t_N}\|\bU^0\|^2+\|\f\|_{\infty}^2\Big),
\end{align}
where $\|\f\|_{\infty}=\|\f\|_{L^{\infty}(\bL^2)},$ and 
$$\Gamma_1= \Big(e^{-\alpha k}\mu-2\big(\frac{1-e^{-\alpha k}}{k}\big)
 \lambda_1^{-1}\Big).
$$
\end{lemma}

\begin{proof}
 Setting $\tbU^n=e^{\alpha t_n}\bU^n,$ we rewrite (\ref{fdbej}), for $\bphi_h\in\bJ_h,$ as
\begin{equation}\label{tfdbej}
 e^{\alpha t_n}(\pt\bU^n,\bphi_h)+\mu a(\tbU^n,\bphi_h)+e^{-\alpha 
t_n}b_h(\tbU^n,\tbU^n,
 \bphi_h)+e^{\alpha t_n}a(q_r^n(\bU),\bphi_h)= (\tf^n,\bphi_h).
\end{equation}
Note that
$$ e^{\alpha t_n}\pt\bU^n= e^{\alpha k}\pt\tbU^n-\Big(\frac{e^{\alpha
   k}-1}{k}\Big)\tbU^n. $$
On substituting this in (\ref{tfdbej}) and then multiplying the resulting equation by
$e^{-\alpha k},$ we obtain
\begin{align}\label{stb001}
 (\pt\tbU^n,\bphi_h)-\Big(\frac{1-e^{-\alpha k}}{k}\Big)(\tbU^n,\bphi_h)+e^{-\alpha k}
 \mu a(\tbU^n,\bphi_h)+e^{-\alpha t_{n+1}}b(\tbU^n,\tbU^n,\bphi_h) \nonumber \\
 +\gamma e^{-\alpha k}\sum_{i=1}^n e^{-(\delta-\alpha)(t_n-t_i)} a(\tbU^i,\bphi_h)
 =e^{-\alpha k}(\tf^n,\bphi_h).
\end{align}
Put $\bphi_h=\tbU^n$ in (\ref{stb001}) and observe that
$$ (\pt\bphi^n,\bphi^n)=\frac{1}{k}(\bphi^n-\bphi^{n-1},\bphi^n) \ge \frac{1}{2k}
   (\|\bphi^n\|^2-\|\bphi^{n-1}\|^2)=\frac{1}{2}\pt\|\bphi^n\|^2, $$
and that the nonlinear term vanishes. Also use $\|\tbU^n\|^2 \le 
\frac{1}{\lambda_1}\|\nabla\tbU^n\|^2$ to obtain
\begin{align}\label{stb002}
 \frac{1}{2}\pt\|\tbU^n\|^2 +& \Big(e^{-\alpha k}\mu-\big(\frac{1-e^{-\alpha 
k}}{k}\big)
 \lambda_1^{-1}\Big)\|\nabla\tbU^n\|^2 \nonumber \\
 +& \gamma e^{-\alpha k}k\sum_{i=1}^n e^{-(\delta-\alpha)(t_n-t_i)} a(\tbU^i,\tbU^n) \le
 e^{-\alpha k}\|\tf^n\|\|\tbU^n\|.
\end{align}
The right-hand side of (\ref{stb002}) can be estimated as
$$ \frac{1}{2}e^{-\alpha k}\mu\|\nabla\tbU^n\|^2+\frac{1}{2\mu\lambda_1}
   e^{-\alpha k}\|\tf^n\|^2, $$
so as to obtain from (\ref{stb002})
\begin{align}\label{stb003}
 \pt\|\tbU^n\|^2 +& \Big(e^{-\alpha k}\mu-2\big(\frac{1-e^{-\alpha k}}{k}\big)
 \lambda_1^{-1}\Big)\|\nabla\tbU^n\|^2 \nonumber \\
 +& 2\gamma e^{-\alpha k}k\sum_{i=1}^n e^{-(\delta-\alpha)(t_n-t_i)} a(\tbU^i,\tbU^n) 
\le \frac{1}{\mu\lambda_1}e^{-\alpha k}\|\tf^n\|^2.
\end{align}
With $0<\alpha<\min\{\delta,\frac{\mu\lambda_1}{2}\},$ we choose $k_0>0$ such that for $0<k<k_0$
$$ 1+\big(\frac{\mu\lambda_1}{2}\big)k\ge e^{\alpha k}. $$
This guarantees that $e^{-\alpha k}\mu-2\big(\frac{1-e^{-\alpha 
k}}{k}\big)\lambda_1^{-1}\ge 0.$ Multiply (\ref{stb003}) by $k$ and then sum over $n=1$ to $N.$ The resulting double sum is non-negative and hence, we obtain
\begin{align} \label{UN-estimate-1}
 \|\tbU^N\|^2+ \Gamma_1 k\sum_{n=1}^N \|\nabla\tbU^n\|^2
 \le \|\bU^0\|^2+\frac{\|\f\|_{\infty}^2}{\mu\lambda_1}e^{-\alpha k} k\sum_{n=1}^N e^{2\alpha t_n}.
\end{align}
Note that using geometric series, we find that
\begin{equation}\label{e-2alpha}
k\sum_{n=1}^N e^{2\alpha t_n}= e^{2\alpha k}\frac{k}{e^{2\alpha k}-1} e^{2\alpha t_N}= e^{2\alpha(k-k^*)} e^{2\alpha t_N},
\end{equation}
for some $k^*$ in $(0,k).$ 
On substituting (\ref{e-2alpha}) in (\ref{UN-estimate-1}),
multiply through out by $e^{-\alpha t_N}$ to complete the rest of the proof.
\end{proof}


In order to obtain uniform (in time) estimate for the discrete solution $\bU^n$
in Dirichlet norm, we introduce the following notation:
\begin{equation}\label{ubeta}
\bU^n_{\beta}=k\sum_{j=1}^n \beta_{nj}\bU^j, n >0;~~\bU^0_{\beta}=0,
\end{equation}
and rewrite (\ref{fdbej}), for $\bphi_h\in\bJ_h,$ as
\begin{equation}\label{fdbej00}
 (\pt\bU^n,\bphi_h)+\mu a(\bU^n,\bphi_h)+b_h(\bU^n,\bU^n,\bphi_h)
 +a(\bU^n_{\beta},\bphi_h)= (\f^n,\bphi_h).
\end{equation}
Note that
\begin{equation}\label{une01}
\bU^n_{\beta}= k\gamma\bU^n+e^{-\delta k}\bU^{n-1}_{\beta},
\end{equation}
and therefore
\begin{align}\label{une02}
\pt\bU^n_{\beta} &=\frac{1}{k}(\bU^n_{\beta}-\bU^{n-1}_{\beta}) =\frac{1}{k}(
k\gamma\bU^n+e^{-\delta k}\bU^{n-1}_{\beta}-\bU^{n-1}_{\beta}) \\
&=\gamma\bU^n-\frac{(1-e^{-\delta k})}{k}\bU^{n-1}_{\beta}. \nonumber
\end{align}

\begin{lemma}\label{unif.est0}
Let $0<\alpha<\min (\delta,\mu\lambda_1/2),~\bU^0=P_h\bu_0$ and $k_0>0$ be such that 
for $0<k<k_0$
$$ 1+\big(\frac{\mu\lambda_1}{2}\big)k\ge e^{\alpha k}. $$
Then, the discrete solution $\bU^n,~n\ge 1$ of (\ref{fdbej}) satisfies the following uniform estimates:
\begin{equation}\label{unif.est01}
\|\bU^n\|^2+\frac{e^{-\delta k}}{\gamma} \|\nabla\bU^n_{\beta}\|^2 \le e^{-2\alpha t_n}\|\bU^0\|^2+ \left(\frac{1-e^{-2\alpha t_n}}{\alpha\mu\lambda_1}\right)\|\f\|^2_{\infty} 
=M_{11}^2,
\end{equation}
and
\begin{equation}\label{unif.est02}
k\sum_{n=m}^{m+l} \big(\mu\|\nabla\bU^n\|^2+\frac{\delta}{\gamma} 
\|\nabla\bU^n_{\beta}\|^2  \big) \le M_{11}^2+\frac{l}{\mu\lambda_1} 
\|\f\|^2_{\infty}=M_{12}^2(l),
\end{equation}
where $\bU^n_{\beta}$ is given by (\ref{ubeta}).
\end{lemma}

\begin{proof}
Take $\bphi_h=\bU^n$ in (\ref{fdbej00}) and from (\ref{une02}), we find that
$$ a(\bU^n_{\beta},\bU^n)= \frac{e^{-\delta k}}{\gamma} a(\bU^n_{\beta},\pt\bU^n_{\beta})
+\frac{(1-e^{-\delta k})}{k\gamma}\|\nabla\bU^n_{\beta}\|^2. $$
Using mean value theorem, we observe that
$$ \frac{(1-e^{-\delta k})}{k} =\delta e^{-\delta k^*} \ge \delta e^{-\delta k},
~~k^*\in (0,k). $$
Therefore, we obtain from (\ref{fdbej00})
\begin{equation}\label{une03}
\pt\big(\|\bU^n\|^2+\frac{e^{-\delta k}}{\gamma}\|\nabla\bU^n_{\beta}\|^2 \big) 
+\mu\|\nabla\bU^n\|^2+\frac{2\delta e^{-\delta k}}{\gamma}\|\nabla\bU^n_{\beta}\|^2 \le 
\frac{1}{\mu\lambda_1}\|\f^n\|^2.
\end{equation}
As $0<\alpha<\min \{\delta,\mu\lambda_1/2\}$, we now find that
\begin{equation}\label{une04}
\pt\big(\|\bU^n\|^2+\frac{e^{-\delta k}}{\gamma}\|\nabla\bU^n_{\beta}\|^2 \big) 
+2\alpha\big(\|\bU^n\|^2+\frac{e^{-\delta k}}{\gamma} 
\|\nabla\bU^n_{\beta}\|^2\big) \le \frac{1}{\mu\lambda_1}\|\f^n\|^2.
\end{equation}
Multiply the inequality (\ref{une04}) by $e^{\alpha_0 t_{n-1}}$ for some $\alpha_0>0$ and 
note that
\begin{eqnarray}\label{discrete01}
\pt(e^{\alpha_0 t_n}\bphi^n) &=& e^{\alpha_0 t_{n-1}}\Big\{\pt\bphi^n +\frac{e^{ 
\alpha_0 k}-1}{k} \bphi^n\Big\} \nonumber \\
&\le & e^{\alpha_0 t_{n-1}}\Big\{\pt\bphi^n +2\alpha\bphi^n\Big\}.
\end{eqnarray}
With the assumption on the time step $k,$ that is, $0<k<k_0,$ and for given $\alpha$, we can always choose $\alpha_0$ such that
\begin{equation}\label{une04a}
1+2\alpha k \ge e^{\alpha_0k}.
\end{equation}
Observe that $\alpha _0 < 2\alpha$. Therefore, we obtain from (\ref{une04})
$$ \pt\Big(e^{\alpha_0 t_n}\Big(\|\bU^n\|^2+\frac{e^{-\delta k}}{\gamma} \|\nabla 
\bU^n_{\beta}\|^2 \Big) \Big) \le\frac{e^{\alpha_0 t_{n-1}}}{\mu\lambda_1}\|\f\|_{\infty}^2. 
$$
Multiply by $k$ and sum over $1$ to $n$ and then multiply the resulting inequality by $e^{-\alpha_0 t_n}.$ Observe that $\bU^0_{\beta}=0$ by definition. This results in the first estimate (\ref{unif.est01}). For the second estimate (\ref{unif.est02}), we multiply (\ref{une03}) by $k,$ sum over $m$ to $m+l$ with $m,l\in \mathcal{N}$ and use (\ref{unif.est01}) to complete the rest of the proof.
\end{proof}

\begin{lemma}\label{unif.est1}
Under the assumptions of Lemma \ref{unif.est0}, the discrete solution $\bU^n,~n\ge 1$ of (\ref{fdbej}) satisfies the following uniform estimates:
\begin{equation}\label{unif.est1a}
\|\nabla\bU^n\|^2+\frac{e^{-\delta k}}{\gamma}\|\td_h\bU^n_{\beta}\|^2 \le K.
\end{equation}
\end{lemma}

\begin{proof}
Set $\bphi_h=-\td_h\bU^n$ in (\ref{fdbej00}) and as in the Lemma \ref{unif.est0}, we now 
obtain
\begin{align}\label{une11}
\pt\big(\|\nabla\bU^n\|^2+\frac{e^{-\delta k}}{\gamma}\|\td_h\bU^n_{\beta}\|^2 
\big) +\mu\|\td_h\bU^n\|^2&+\frac{2\delta}{\gamma}\|\nabla\bU^n_{\beta}\|^2 \le \|\f^{n}\|\|\td_h\bU^n\| \nonumber \\
&+|b_h(\bU^n,\bU^n,-\td_h\bU^n)|.
\end{align}
Use Lemma \ref{nonlin} to arrive at
\begin{align}\label{une12}
 \pt\left(\|\nabla\bU^n\|^2+\frac{e^{-\delta k}}{\gamma}\|\td_h\bU^n_{\beta}\|^2 
\right)+& \frac{4\mu} {3}\|\td_h\bU^n\|^2+\frac{2\delta}{\gamma} 
\|\td_h\bU^n_{\beta}\|^2 \nonumber \\
&\le \frac{3}{\mu} \|\f\|_{\infty}^2 +(\frac{9/2}{\mu})^3 M_{11}^2\|\nabla\bU^n\|^4. 
\end{align}
For some $\alpha_0>0$, we find that
\begin{equation}\label{une13} 
 \alpha_0\|\nabla\bU^n\|^2 \le \frac{\mu}{3}\|\td_h\bU^n\|^2 
+\frac{3}{4\mu}\alpha_0^2\|\bU^n\|^2.
\end{equation}
Define
\begin{equation}\label{une14}
 g^n= \min\Big\{\alpha_0+\mu\lambda_1-(\frac{9}{2\mu})^3 M_{11}^2
 \|\nabla\bU^n\|^2, ~2\delta\Big\}.
\end{equation}
With $E^n :=\|\nabla\bU^n\|^2+\frac{e^{-\delta k}}{\gamma} \|\td_h\bU^n_{\beta}\|^2$, we rewrite (\ref{une12}) as
\begin{equation}\label{une15}
 \pt E^n+g^nE^n \le \frac{3}{\mu}\|\f\|_{\infty}^2+\frac{3}{4\mu}\alpha_0^2\|\bU^n\|^2 
=K_{11}. 
\end{equation}
Let $\{n_{i}\}_{i\in\mathbb{N}}$ and $\{\bar{n}_{i}\}_{i\in\mathbb{N}}$ be two subsequences of natural numbers such that
$$ g^{n_{i}}=\alpha_0+\mu\lambda_1-(\frac{9}{2\mu})^3 M_{11}^2\|\nabla\bU^{n_i}\|^2,
~~g^{\bar{n}_{i}}=2\delta,~\forall i. $$
If for some $n$,
$$ g^n=\alpha_0+\mu\lambda_1-(\frac{9}{2\mu})^3 M_{11}^2\|\nabla\bU^n\|^2=2\delta $$
then without loss of generality, we assume that $n\in \{\bar{n}_{i}\} $ so as to make
the two subsequence $\{n_{i}\}$ and $\{\bar{n}_{i}\}$ disjoint. Now for $m,l\in\mathbb{N}$, we write
\begin{align}\label{une16}
k\sum_{n=m}^{m+l}g^n &= k\sum_{n=m_1}^{m_{l_1}} g^n+k\sum_{n=\bar{m}_1}^{\bar{m}_{l_2}}
g^n \nonumber \\
&=k\sum_{n=m_1}^{m_{l_1}} \left(\alpha_0+\mu \lambda_1-(\frac{9}{2\mu})^3 M_{11}^2  \|\nabla\bU^n\|^2\right)+k\sum_{n=\bar{m}_1}^{\bar{m}_{l_2}} 2\delta.
\end{align}
Here, $m_1,m_2,\cdots,m_{l_1} \in\{n_{i}\}\cap \{m,m+1,\cdots,
m+l\}$ and $\bar{m}_1,\bar{m}_2,\cdots,\bar{m}_{l_2} \in\{\bar{n}_{i}\}\cap \{m,m+1,
\cdots,m+l\}$ such that $l_1+l_2=l+1.$ Note that $l_1$ or $l_2$ could be $0$.
Using Lemma \ref{unif.est0}, we observe that
\begin{align*}
(\frac{9}{2\mu})^3k\sum_{n=m}^{m+l} M_{11}^2 \|\nabla\bU^n\|^2 \le 
\frac{9^3M_{11}^2}{2^3\mu^3}k\sum_{n=m}^{m+l} \|\nabla\bU^n\|^2 \le 
\frac{9^3M_{11}^2}{2^3\mu^4}M_{12}^2(l)=K_{12}(l).
\end{align*}
Therefore, from (\ref{une16}), we find that
\begin{align*}
k\sum_{n=m}^{m+l}g^n \ge (kl_1)(\alpha_0+\mu \lambda_1)-K_{12}(l_1)+2\delta (kl_2).
\end{align*}
We choose $\alpha_0$ such that $(kl_1)(\alpha_0+\mu \lambda_1)-K_{12}(l_1)=2\delta (kl_1)$
to arrive at
 \begin{equation}\label{une17}
k\sum_{n=m}^{m+l}g^n \ge 2\delta t_{l+1}.
\end{equation}
By definition of $g^n$, we have equality in (\ref{une17}) and in fact, $g^n=2\delta$.
Now from (\ref{une15}), we obtain
$$ \pt E^n+2\delta E^n \le K_{11}. $$
As in (\ref{discrete01}), we can choose $0<\alpha_{01} < \alpha \le \delta$ such that
$$ \pt(e^{\alpha_{01}t_n}E^n) \le e^{\alpha_{01}t_{n-1}}(\pt E^n+2\delta E^n)
\le K_{11}e^{\alpha_{01}t_{n-1}}. $$
Multiply by $k$ and sum over $1$ to $n$. Observe that $E^0=\|\nabla\bU^0\|^2$. Finally,
multiply the resulting inequality by $e^{-\alpha_{01}t_n}$ to find that
$$ E^n \le e^{-\alpha_{01}t_n}\|\nabla\bU^0\|^2+K. $$
This completes the rest of the proof.
\end{proof}

\begin{remark}
As a consequence of the Lemma 4.3, the following {\it a priori} bound is valid:
\begin{equation}
\label{h2-bound}
\tau*(t_n)  \|\td_h\bU^n\|^2 \leq K. 
\end{equation}
\end{remark}

\section{A Priori Error Estimate}
\se

In this section, we discuss error estimate of the backward Euler method for the Oldroyd 
model (\ref{om})-(\ref{ibc}). For the error analysis, we set, for fixed $n\in\mathbb{N}, ~1< 
n\le N,~\e_n=\bU^n -\bu_h(t_n)=\bU^n-\bu_h^n.$
We now rewrite (\ref{dwfj}) at $t=t_n$ and subtract the resulting one from (\ref{fdbej}) to 
obtain
\begin{align}\label{eebe}
 (\pt\e_n,\bphi_h)+ \mu a(\e_n,\bphi_h)+a(q^n_r(\e),\bphi_h) = E^n (\bu_h)(\bphi_h) 
+\ve_a^n(\bu_h)(\bphi_h)+\Lambda^n_h(\bphi_h),
\end{align}
where,
\begin{eqnarray}
 E^n(\bu_h)(\bphi_h) &=&  (\bu_{ht}^n,\bphi_h)-(\pt\bu_h^n,\bphi_h) = (\bu_{ht}^n,\bphi_h) 
-\frac{1}{k}\int_{t_{n-1}}^{t_n}(\bu_{hs}, \bphi_h)~ds \nonumber \\
 &=& \frac{1}{2k}\int_{t_{n-1}}^{t_n} (t-t_{n-1})(\bu_{htt},\bphi_h)dt, \label{R1be} \\
 \ve_a^n(\bu_h)(\bphi_h) &=&  a(\uhb(t_n), \bphi_h)ds -a(q_r^n(\bu_h), \bphi_h)=a(\ve_r^n(\bu_h),\bphi_h), \label{ver}
\end{eqnarray}
and
\begin{align}\label{dLbe}
 \Lambda^n_h(\bphi_h) &= b(\bu_h^n,\bu_h^n,\bphi_h)-b(\bU^n,\bU^n,\bphi_h)
\nonumber \\
& = -b(\bu_h^n,\e_n,\bphi_h)-b(\e_n,\bu_h^n,\bphi_h)-b(\e_n,\e_n,\bphi_h).
\end{align}
In order to dissociate the effect of nonlinearity, we first linearized the discrete problem (\ref{fdbej}), and introduce $\{\bV^n\}_{n\ge 1}\in\bJ_h$ as solutions of the following linearized problem:
\begin{equation}\label{fdbejv} 
 (\pt\bV^n,\bphi_h)+\mu a(\bV^n,\bphi_h)+a(q_r^n(\bV),\bphi_h)= (\f^n,\bphi_h)
 -b(\bu_h^n,\bu_h^n,\bphi_h)~~~\forall\bphi_h\in\bJ_h,
\end{equation}
given $\{\bU^n\}_{n\ge 1}\in\bJ_h$ as solution of (\ref{fdbej}). It is easy to check the existence and uniqueness of 
$\{\bV^n\}_{n\ge 1}\in\bJ_h.$

We now split the error as:
\begin{align}\label{errsplit}
\e_n:=\bU^n-\bu_h^n = (\bU^n-\bV^n)-(\bu_h^n-\bV^n) =: \bta_n-\bxi_n.
\end{align}
The following equations are satisfied by $\bxi_n$ and $\bta_n,$ respectively:
\begin{align}\label{eebelin}
 (\pt\bxi_n,\bphi_h)+& \mu a(\bxi_n,\bphi_h)+a(q^n_r(\bxi),\bphi_h)
 = -E^n(\bu_h)(\bphi_h)-\ve_a^n(\bu_h)(\bphi_h)
\end{align}
and
\begin{align}\label{eebenl}
 (\pt\bta_n,\bphi_h)+& \mu a(\bta_n,\bphi_h)+a(q^n_r(\bta),\bphi_h) 
 =\Lambda^n_h(\bphi_h).
\end{align}
Below, we prove the following Lemma for our subsequent use.
\begin{lemma}\label{e-ve}
Let $ r,s\in \{0,1\},\tau_i=\min\{1,t_i\}$ and $\alpha$ as defined in  Lemma 4.1. Then,
with $E^n$ and $\ve_a^n$ defined, respectively, as (\ref{R1be}) and (\ref{ver}), 
the following estimate holds 
for $n=1,\cdots,N$ and for $\{\bphi_h^i\}_i$ in $\bJ_h$:
\begin{eqnarray}\label{eve1}
&& 2k\sum_{i=1}^n \tau_i^s e^{2\alpha (t_i-t_n)}\Big(E^i(\bu_h)(\bphi_h^i)+\ve_a^i(\bu_h)
(\bphi_h^i)\Big) \\
&&\le Kk^{(1+s-r)/2}(1+\log\frac{1}{k})^{(1-r)/2}\left({k\sum_{i=1}^n \tau_i^s e^{2\alpha (t_i-t_n)}\|\bphi_h^i\|_{1-r}^2}\right)^{1/2}. \nonumber
\end{eqnarray}
\end{lemma}

\begin{proof}
From (\ref{R1be}), we observe that
\begin{align*}
2k &\sum_{i=1}^n \tau_i^s e^{2\alpha (t_i-t_n)} E^i(\bu_h)(\bphi_h^i) \\ \nonumber 
\le &  \left[k^{-1}\sum_{i=1}^n \Big(\int_{t_{i-1}}^{t_i} \tau_i^{s/2} e^{\alpha (t_i-t_n)}(t-t_{i-1})\|\bu_{htt}\|_{r-1}\,dt \Big)^2\right]^{1/2}
\left[k\sum_{i=1}^n \tau_i^s e^{2\alpha (t_i-t_n)}\|\bphi_h^i\|_{1-r}^2
\right]^{1/2}. \nonumber
\end{align*}
Using (\ref{dth13}), we find 
\begin{align}\label{eve01a}
& \left[{k^{-1}\sum_{i=1}^n \Big(\int_{t_{i-1}}^{t_i} \tau_i^{s/2} e^{\alpha (t_i-t_n)}(t-t_{i-1})\|\bu_{htt}\|_{r-1}dt \Big)^2}\right]^{1/2} \nonumber \\
\le & \left[{k^{-1}\sum_{i=1}^n \int_{t_{i-1}}^{t_i} \tau_i^s \tau^{-(r+1)}(t-t_{i-1})^2
e^{2\alpha (t_i-t)}\,dt}\right]^{1/2}\left[{e^{-2\alpha t_n}\int_0^{t_n} \tau^{(r+1)} e^{2\alpha t}\|\bu_{htt}\|^2_{r-1}\,dt}\right]^{1/2} \\
\le & Ke^{\alpha k}\left[{k^{-1}\sum_{i=1}^n \int_{t_{i-1}}^{t_i} \tau_i^s \tau(t)^{-(r+1)}(t-t_{i-1})^2}\,dt\right]^{1/2}. \nonumber
\end{align}
It is now easy to calculate the remaining part for various values of $r,s.$ For the sake of completeness, we present below the case when $r=s=0.$
\begin{align*}
\sum_{i=1}^n \int_{t_{i-1}}^{t_i} t^{-1}(t-t_{i-1})^2 dt
& \le \int_0^k t\,dt+k^2\sum_{i=2}^n \int_{t_{i-1}}^{t_i} t^{-1}\,dt \\
& \le Kk^2(1+\log\frac{1}{k}).
\end{align*}
This completes the proof of the first half. For the remaining part, we observe from (\ref{ver}) and
(\ref{errrr}) that
\begin{align}\label{eve02}
& 2k\sum_{i=1}^n \tau_i^s e^{2\alpha (t_i-t_n)} \ve_a^i(\bu_h)(\bphi_h^i) \le \left[{k\sum_{i=1}^n \tau_i^s e^{2\alpha (t_i-t_n)}\|\bphi_h^i\|_{1-r}^2}\right]^{1/2}~\times \\
& \left[{4k\sum_{i=1}^n \Big(\sum_{j=1}^i \int_{t_{j-1}}^{t^j} \tau_i^{s/2}e^{\alpha (t_i-t_n)}(t-t_{j-1})
\beta(t_i-t)\{\delta\|\bu_h\|_{r+1}+\|\bu_{ht}\|_{r+1}
\}\,dt\Big)^2}\right]^{1/2}. \nonumber 
\end{align}
In Lemma \ref{dth2}, we find that the estimates of $\|\bu_{htt}\|_{r-1}$ and $\|\bu_{ht}\|_{r+1}$ are similar, in fact, the powers of $t_i$ are same. Therefore,the right-hand side of (\ref{eve02}) involving $\|\bu_{ht}\|_{r+1}$ can be estimated
similarly as in (\ref{eve01a}). The terms involving $\|\bu_h\|_{r+1}$ are clearly easy to estimate. But for the sake of completeness, we provide  the case, when $ r=s=0.$
 \begin{align*}
& 4\delta^2 k\sum_{i=1}^n \Big(\sum_{j=1}^i \int_{t_{j-1}}^{t^j} e^{\alpha (t_i-t_n)}
(t-t_{j-1})\beta(t_i-t) \|\nabla\bu_h\|~dt\Big)^2 \nonumber \\
\le & ~4\gamma^2\delta^2 e^{-2\alpha t_n} k^3\sum_{i=1}^n e^{-2(\delta-\alpha)t_i} \Big(\sum_{j=1}^i \int_{t_{j-1}}^{t^j}e^{(\delta-\alpha)t}\|\nabla\hbu_h\|~dt\Big)^2 \nonumber \\
 \le & ~4\gamma^2\delta^2 e^{-2\alpha t_n} k^3\sum_{i=1}^n e^{-2(\delta-\alpha)t_i} \Big(\int_0^{t_i} e^{2(\delta-\alpha)s}ds\Big) \Big(\int_0^{t_i} \|\nabla\hbu_h(s)\|^2 ds\Big) \nonumber \\
 \le & \frac{2\gamma^2\delta^2}{2 (\delta-\alpha)} e^{-2\alpha t_n}k^3 \sum_{i=1}^n e^{2(\delta-\alpha)k}\big(Ke^{2\alpha t_i}\big) \le Kk^3 e^{2\delta k}.
 \end{align*}
This completes the rest of the proof.
\end{proof}
 
\begin{lemma}\label{pree}
 Assume (${\bf A1}$)-(${\bf A2}$) and a spatial discretization scheme that satisfies
conditions (${\bf B1}$)-(${\bf B2}$) and (${\bf B4}$). Let $0 < \alpha < \min 
\big\{\delta, \mu\lambda_1\big\},$ and
$$ 1+(\mu\lambda_1)k > e^{2\alpha k} $$
which holds for $0<k<k_0,~k_0>0.$ Further, assume that $\bu_h(t)$ and $\bV^n$ satisfy
 (\ref{dwfj}) and (\ref{fdbejv}), respectively. Then, there is a positive
constant $K$ such that
\begin{eqnarray} \label{pree1}
 \|\bxi_n\|^2&+& e^{-2\alpha t_n} k\sum_{i=1}^n e^{2\alpha t_i} \|\bxi_i\|_1^2 \le  Kk\big(1+\log\frac{1}{k}\big), \\
 \|\bxi_n\|_1^2&+& k\sum_{i=1}^n \{\|\bxi_i\|_2^2+\|\pt\bxi_i\|^2\} \le K.  \label{pree2}
\end{eqnarray}
\end{lemma}

\begin{proof}
For $n=i,$ we put $\bphi_h=\bxi_i$ in (\ref{eebelin}) and with the observation
$$ (\pt\bxi_i,\bxi_i)=\frac{1}{2k}(\bxi_i-\bxi_{i-1},\bxi_i) \ge \frac{1}{2k}(\|\bxi_i\|^2
   -\|\bxi_{i-1}\|^2)=\frac{1}{2}\pt\|\bxi_i\|^2, $$
we find that
\begin{equation}\label{pree01}
 \pt\|\bxi_i\|^2+2\mu\|\nabla\bxi_i\|^2+a(q^i_r(\bxi),\bxi_i) \le -2E^i(\bu_h)(\bxi_i) -2\ve_a^i(\bu_h)(\bxi_i).
\end{equation}
Multiply (\ref{pree01}) by $ke^{2\alpha t_i}$ and sum over $1\le i\le n\le N$ to obtain
\begin{align}\label{pree02}
\|\tbxi_n\|^2 -\sum_{i=1}^{n-1} (e^{2\alpha k}-1)\|\tbxi_i\|^2+2\mu k\sum_{i=1}^n 
\|\nabla\tbxi_i\|^2 & \le -2k\sum_{i=1}^n e^{2\alpha t_i}\Big\{E^i(\bu_h)(\bxi_i) +\ve_a^i(\bu_h)(\bxi_i)\Big\} \nonumber \\
& \le \mu k\sum_{i=1}^n \|\nabla\tbxi_i\|^2+Kk \big(1+\log\frac{1}{k}\big)
 e^{2\alpha t_{n+1}}.
\end{align}
Recall that $\tilde{v}(t)=e^{\alpha t}v(t)$.
Note that we have dropped the quadrature term on the left hand-side of (\ref{pree01}) after summation as it is non-negative. Finally, we have used Lemma \ref{e-ve} for $s=r=0$.
We note that for $0<k<k_0$
$$ \mu-\frac{e^{2\alpha k}-1}{k\lambda_1} >0, $$
and hence,
\begin{align}\label{pree06}
\|\tbxi_n\|^2 &+ (\mu-\frac{e^{2\alpha k}-1}{k\lambda_1}) k\sum_{i=1}^n 
\|\nabla\tbxi_i\|^2\le Kk \big(1+\log\frac{1}{k}\big)e^{2\alpha t_{n+1}}.
\end{align}
Multiply (\ref{pree06}) by $e^{-2\alpha t_n}$ to establish (\ref{pree1}).
Next, for $n=i,$ we put $\bphi_h=-\td_h\bxi_i$ in (\ref{eebelin}) and follow
as above to obtain the first part of (\ref{pree2}), that is,
$$ \|\bxi_n\|_1^2+k\sum_{i=1}^n \|\bxi_i\|_2^2 \le K. $$
Here, we have used (\ref{eve1}) for $s=0, r=1$ with $\alpha=0$  replacing
$\bphi_h^{i}$ by $\td_h\bxi_i.$

Finally, for $n=i,$ we put $\bphi_h=\pt\bxi_i$ in (\ref{eebelin}) to find that
\begin{align}\label{pree07}
2\|\pt\bxi_i\|^2+\mu\pt\|\bxi_i\|_1^2 \le -2a(q_r^i(\bxi),\pt\bxi_i)
-2E^i(\bu_h)(\pt\bxi_i) -2\ve_a^i(\bu_h)(\pt\bxi_i).
\end{align}
Multiply (\ref{pree07}) by $ke^{2\alpha t_i}$ and sum over $1\le i\le n\le N$.
As has been done earlier, we can estimate the last two resulting terms on the right-hand side of (\ref{pree07}) using (\ref{eve1}) for $r=s=0$ as
$$ \frac{k}{2} \sum_{i=1}^n e^{2\alpha t_i}\|\pt\bxi_i\|^2+K. $$
The only difference is that the resulting double sum (the term involving $q_r^i$)
is no longer non-negative and hence, we need to estimate it. Note that
\begin{align}\label{pree08}
2k\sum_{i=1}^n e^{2\alpha t_i} a(q_r^i(\bxi),\pt\bxi_i)=2\gamma k^2\sum_{i=1}^n
\sum_{j=1}^i e^{-(\delta-\alpha)(t_i-t_j)} a(\tbxi_j,e^{\alpha t_i}\pt\bxi_i) \\
\le \frac{k}{2} \sum_{i=1}^n e^{2\alpha t_i}\|\pt\bxi_i\|^2
+K(\gamma) k\sum_{i=1}^n \Big(k\sum_{j=1}^i e^{-(\delta-\alpha)(t_i-t_j)} 
\|\td_h\tbxi_j\|\Big)^2. \nonumber
\end{align}
Using change of variable and change of order of double sum, we obtain
\begin{align*}
I & := K(\gamma) k\sum_{i=1}^n \Big(k\sum_{j=1}^i e^{-(\delta-\alpha)(t_i-t_j)} 
\|\td_h\tbxi_j\|\Big)^2 \\
& \le K(\gamma) k\sum_{i=1}^n \Big(k\sum_{j=1}^i e^{-(\delta-\alpha)(t_i-t_j)}
\Big) \Big(k\sum_{j=1}^i e^{-(\delta-\alpha)(t_i-t_j)}\|\td_h\tbxi_j\|^2\Big) \\
& \le K(\alpha,\gamma) e^{(\delta-\alpha)k} k^2 \sum_{i=1}^n k\sum_{j=1}^i e^{-(\delta-\alpha)(t_i-t_j)}\|\td_h\tbxi_j\|^2.
\end{align*}
Introduce $l=i-j$ to find that
\begin{align*}
I& \le K(\alpha,\gamma) e^{(\delta-\alpha)k}k^2\sum_{i=1}^n k\sum_{l=i-1}^0 e^{-(\delta-\alpha) t_l}  \|\td_h\tbxi_{i-l}\|^2~~~\mbox{for} ~l=i-j\\
&= K(\alpha,\gamma) e^{(\delta-\alpha)k}k^2\sum_{i=1}^n k\sum_{l=1}^{i} e^{-(\delta-\alpha)t_{l-1}}\|\td_h\tbxi_{i-l+1}\|^2.
\end{align*}
With change of summation, we now arrive at 
\begin{align}\label{pree09}
I&\le  K(\alpha,\gamma) e^{(\delta-\alpha)k}k^2\sum_{l=1}^n k\sum_{i=l}^{n} e^{-(\delta-\alpha)t_{l-1}}\|\td_h\tbxi_{i-l+1}\|^2 \nonumber \\
&= K(\alpha,\gamma) e^{(\delta-\alpha)k}k^2 \sum_{l=1}^n k\sum_{j=1}^{n-l+1} e^{-(\delta-\alpha)t_{l-1}}\|\td_h\tbxi_j\|^2~~~\mbox{for} ~j=i-l+1 \nonumber \\
& \le K(\alpha,\gamma) e^{(\delta-\alpha)k}k \Big(k\sum_{l=1}^{n-1} e^{-(\delta-\alpha)t_l}\Big)\Big(k\sum_{j=1}^n \|\td_h\tbxi_j\|^2\Big) \le K.
\end{align}
Combining (\ref{pree08})-(\ref{pree09}), we find that
\begin{align*}
2k\sum_{i=1}^n e^{2\alpha t_i} a(q_r^i(\bxi),\pt\bxi_i)\le \frac{k}{2}
\sum_{i=1}^n e^{2\alpha t_i}\|\pt\bxi_i\|^2+K.
\end{align*}
Therefore, we obtain
\begin{align}
k\sum_{i=1}^n e^{2\alpha t_i}\|\pt\bxi_i\|^2+\mu\|\tbxi_n\|_1^2 \le K+\mu
k \sum_{i=1}^{n-1} \frac{ (e^{2\alpha k}-1)}{k}\|\tbxi_i\|_1^2.
\end{align}
Use (\ref{pree1}) and the fact that $(e^{2\alpha k}-1)/k \le K(\alpha)$ to
complete the rest of the proof.
\end{proof}

\begin{remark}
We note that the restriction on $k$, that is $0<k<k_0$ is not same in the Lemmas \ref{stb},
and \ref{pree}. Therefore, we take minimum of the $k_0$'s from Lemmas
\ref{stb} and \ref{pree} and denote it as $k_{00}$, then for all $k$ satisfying $0<k<k_{00}$, all the result should hold.
\end{remark}

Analogous to the semi-discrete case, we resort to duality argument to obtain optimal 
$L^2(\bL^2)$ estimate. Consider the following backward problem: For a given
$\bW_n$ and $\g_i,$ let $\bW_i,~n\ge i\ge 1 $ satisfy 
\begin{equation}\label{bwprob}
(\bphi_h,\pt\bW_i)-\mu a(\bphi_h, \bW_i)-k\sum_{j=i}^n \beta(t_j-t_i) a(\bphi_h, \bW_j) 
=(\bphi_h, e^{2\alpha t_i} \g_i),\bphi_h\in \bJ_h.
\end{equation}
The following {\it a priori} estimates are easy to derive.
\begin{lemma}\label{bwest}
Let the assumptions (${\bf A2}$), (${\bf B1}$), (${\bf B2}$) and (${\bf B4}$) hold. Then,
for $0<k<k_0$, the following estimates hold under appropriate assumptions on $\bW_n$ and $g$:
\begin{equation*}
\|\bW_0\|_r^2+ k\sum_{i=1}^n e^{-2\alpha t_i}\{\|\bW_i\|_{r+1}+ \|\pt\bW_i\|_{r-1}\} \le K\big\{\|\bW_n\|_r^2+k\sum_{i=1}^n e^{2\alpha t_i} \|\g_i\|_{r-1}^2\big\},
\end{equation*}
where $r\in \{0,1\}$.
\end{lemma}

\begin{lemma}\label{neg}
Under the assumptions of Lemma \ref{bwest}, the following estimate holds:
\begin{equation}
e^{-2\alpha t_n} k \sum_{i=1}^n e^{2\alpha t_i} \|\bxi_i\|^2 \le Kk^2. \label{neg1}
\end{equation}
\end{lemma}

\begin{proof}
With
$$ \bW_n=(-\td_h)^{-1}\bxi_n,~~\g_i=\bxi_i~\forall i$$
we choose $\bphi_h=\bxi_i$ in (\ref{bwprob}) and use (\ref{eebelin}) to obtain
\begin{align}\label{neg01}
e^{2\alpha t_i}\|\bxi_i\|^2 &=(\bxi_i,\pt\bW_i)-\mu a(\bxi_i, \bW_i)-k\sum_{j=i}^n 
\beta(t_j-t_i) a(\bxi_i, \bW_j) \nonumber \\
&= \pt (\bxi_i,\bW_i)-(\pt\bxi_i,\bW_{i-1})-\mu a(\bxi_i, \bW_i)-k\sum_{j=i}^n \beta(t_j-t_i) 
a(\bxi_i, \bW_j) \nonumber \\
&= \pt (\bxi_i,\bW_i)+k(\pt\bxi_i,\pt\bW_i)+k\sum_{j=1}^i \beta(t_i-t_j) a(\bxi_j, \bW_i)+E^i(\bu_h)(\bW_i) \nonumber \\
&+\ve_a^i(\bu_h)(\bW_i)-k\sum_{j=i}^n \beta(t_j-t_i) a(\bxi_i, \bW_j). 
\end{align}
Multiply (\ref{neg01}) by $k$ and sum over $1\le i\le n$. Observe that the resulting two 
double sums cancel out (change of order of double sum). Therefore, we find that
\begin{align}\label{neg02}
k\sum_{i=1}^n e^{2\alpha t_i}\|\bxi_i\|^2+\|\bxi_n\|_{-1}^2 =k\sum_{i=1}^n \big[ 
k(\pt\bxi_i,\pt\bW_i)+E^i(\bu_h)(\bW_i)+\ve_a^i(\bu_h) (\bW_i)\big].
\end{align}
From (\ref{R1be}), we observe that
\begin{align}\label{neg03}
k\sum_{i=1}^n E^i(\bu_h)(\bW_i) \le k\sum_{i=1}^n \frac{1}{2k}\int_{t_{i-1}}^{t_i} 
(s-t_{i-1})\|\bu_{hss}\|_{-2}\|\bW_i\|_2 \nonumber \\
\le \frac{k}{4} e^{\alpha k}\Big(\int_0^{t_n} e^{2\alpha s}\|\bu_{hss}\|_{-2}^2 
ds\Big)^{1/2}\Big(k\sum_{i=1}^n e^{-2\alpha t_i}\|\bW_i\|_2^2\Big)^{1/2}.
\end{align}
Similar to (\ref{pree06}), we obtain
\begin{align}\label{neg04}
k\sum_{i=1}^n \ve_a^i(\bu_h)(\bW_i)\le K\Big(k^3\sum_{i=1}^n \int_0^{t_i} e^{2\alpha s} 
(\|\bu_h\|^2+\|\bu_{hs}\|^2)~ds\Big)^{1/2}\Big(k\sum_{i=1}^n e^{-2\alpha t_i} 
\|\bW_i\|_2^2\Big)^{1/2},
\end{align}
and
\begin{align}\label{neg05}
k\sum_{i=1}^n k(\pt\bxi_i,\pt\bW_i)\le k\Big(k\sum_{i=1}^n e^{2\alpha t_i} \|\pt \bxi_i\|^2\Big)^{1/2}\Big(k\sum_{i=1}^n e^{-2\alpha t_i}\|\pt\bW_i\|^2\Big)^{1/2}.
\end{align}
Incorporating (\ref{neg03})-(\ref{neg05}) in (\ref{neg02}), and using Lemmas \ref{dth2} and \ref{bwest}, we find that
\begin{align}\label{neg07}
k\sum_{i=1}^n e^{2\alpha t_i}\|\bxi_i\|^2+\|\bxi_n\|_{-1}^2 \le Kk^2 e^{2\alpha t_n}.
\end{align}
\end{proof}

Due to the non-smooth initial data, we need some intermediate results involving the
``hat operator'' which is defined as
\begin{equation}\label{sum0}
{\hat{\bphi}}_h^n := k\sum_{i=1}^n \bphi_h^i.
\end{equation}
This can be considered as discrete integral operator. We first observe, using (\ref{sumbp}), 
that
\begin{align*}
& k\sum_{j=1}^i \beta(t_i-t_j) \bphi_j = \gamma e^{-\delta t_i}k\sum_{j=1}^i e^{\delta t_j} 
\bphi_j \\
= & \gamma e^{-\delta t_i} \Big\{e^{\delta t_i}\hat{\bphi}_i- k\sum_{j=1}^{i-1} 
(\frac{e^{\delta t_{j+1}}-e^{\delta t_j}}{k}) \hat{\bphi}_j \Big\}= \pt^i \Big\{k 
\sum_{j=1}^i \beta(t_i-t_j) \hat{\bphi}_j \Big\}.
\end{align*}
Here $\pt^i$ means the difference formula with respect to $i$. Now rewrite the equations 
(\ref{eebelin}) (for $n=i$) as follows: 
\begin{align}\label{eebelin1}
 (\pt\bxi_i,\bphi_h)+& \mu a(\bxi_i,\bphi_h)+\pt^i \Big\{k\sum_{j=1}^i \beta(t_i-t_j)  a 
(\hbxi_j,\bphi_h) \Big\} = -E^i(\bu_h)(\bphi_h)-\ve_a^i(\bu_h)(\bphi_h).
\end{align}
We multiply (\ref{eebelin1}) by $k$ and sum over $1$ to $n$.  Using the fact that 
$\pt\hbxi_n=\bxi_n,$ we observe that
\begin{align}\label{eebelin1i}
 (\pt\hbxi_n,\bphi_h)+ \mu a(\hbxi_n,\bphi_h)+ a(q^n_r(\hbxi),\bphi_h) = -k\sum_{i=1}^n 
\big(E^i(\bu_h)(\bphi_h)+\ve_a^i(\bu_h)(\bphi_h)\big).
\end{align}

\begin{lemma}\label{ieens}
 Under the assumptions of Lemma \ref{pree}, the following estimate holds:
\begin{align}\label{ieens1}
\|\hbxi_n\|^2+ e^{-2\alpha t_n} k\sum_{i=1}^n e^{2\alpha t_i}\|\nabla\hbxi_i\|^2 \le  Kk^2(1+\log\frac{1}{k}).
\end{align}
\end{lemma}

\begin{proof}
Choose $\bphi_h=\hbxi_i$ in (\ref{eebelin1i}) for $n=i$, multiply by $k e^{2\alpha t_i}$ and 
then sum over $1\le i\le n$. We drop the third term on the left hand-side of the resulting 
inequality due to non-negativity.
\begin{align}\label{ieens01}
e^{2\alpha t_n}\|\hbxi_n\|^2+\mu k\sum_{i=1}^n e^{2\alpha t_i}\|\nabla\hbxi_i\|^2 \le 
k\sum_{i=1}^n e^{2\alpha t_i} k\sum_{j=1}^i \big(|E^j(\bu_h)(\hbxi_i)| 
+|\ve_a^j(\bu_h)(\hbxi_i)|\big).
\end{align}
From (\ref{R1be}), we find that
$$ k\sum_{j=1}^i |E^j(\bu_h)(\hbxi_i)| \le \frac{1}{2}\big(\sum_{j=1}^i \int_{t_{j-1}}^{t_j} (s-t_{j-1}) 
\|\bu_{hss}\|_{-1}ds \big)\|\nabla\hbxi_i\|. $$
Similar to the proof of Lemma \ref{e-ve} , we split the sum in $j=1$ and the rest to obtain
\begin{align}\label{ieens02}
k\sum_{j=1}^i |E^j(\bu_h)(\hbxi_i)| \le Kk(1+\frac{1}{2}\log\frac{1}{k}) e^{-\alpha k} 
\|\nabla\hbxi_i\|.
\end{align}
Therefore,
\begin{align}\label{ieens03}
k\sum_{i=1}^n e^{2\alpha t_i} k\sum_{j=1}^i |E^j(\bu_h)(\hbxi_i)| \le \frac{\mu}{4} k
 \sum_{i=1}^n e^{2\alpha t_i}\|\nabla\hbxi_i\|^2+Kk^2(1+\log\frac{1}{k}) e^{2\alpha t_n}.
\end{align}
Similarly
\begin{align}\label{ieens04b}
k\sum_{i=1}^n e^{2\alpha t_i} k\sum_{j=1}^i |\ve^j_a(\bu_h)(\hbxi_i)| \le
 \frac{\mu}{4} k \sum_{i =1}^n e^{2\alpha t_i}\|\nabla\hbxi_i\|^2+Kk^2
 (1+\log\frac{1}{k}) e^{2\alpha t_n}.
\end{align}
Incorporate (\ref{ieens03})-(\ref{ieens04b}) in (\ref{ieens01}) to complete 
the rest of the proof.
\end{proof}
\noindent
We are now in a position to estimate $L^{\infty}(\bL^2)$-norm of $\bxi_n$.
\begin{theorem}\label{l2eebxi}
 Under the assumptions of Lemma \ref{pree}, the following holds:
\begin{equation}\label{l2eebxi1}
t_n \|\bxi_n\|^2 +e^{-2\alpha t_n} k\sum_{i=1}^n \sigma_i\|\nabla\bxi_i\|^2 \le 
Kk^2(1+\log \frac{1}{k}),
\end{equation}
where $\sigma_i=t_i e^{2\alpha t_i}$.
\end{theorem}
\begin{proof}
 From (\ref{eebelin}) with $n=i$ and $\bphi_h=\sigma_i\bxi_i$, we obtain
\begin{align}\label{teebelin} 
\pt (\sigma_i\|\bxi_i\|^2)-e^{2\alpha k}\Big\{\|\tbxi_{i-1}\|^2+(\frac{1-e^{-2\alpha k}
 }{k})\sigma_{i-1}\|\bxi\|_{i-1}^2\Big\}+2\mu \sigma_i\|\nabla\bxi_i\|^2 \nonumber 
\\
+2\sigma_i a(q_{r}^i(\bxi), \bxi_i) \le -2E^i(\bu_h)(\sigma_i\bxi_i)- 2\ve_{a}^i
(\bu_h) (\sigma_i \bxi_i).
\end{align}
We multiply (\ref{teebelin}) by $k$ and sum it over $1\le i\le n$ to find that
\begin{align}\label{l2eebxi01}
\sigma_n\|\bxi_n\|^2 + (2\mu-\frac{e^{2\alpha k}-1}{k\lambda_1}) k\sum_{i=1}^n 
\sigma_i\|\nabla\bxi_i\|^2\le e^{2\alpha k} k\sum_{i=2}^{n-1}\|\tbxi_i\|^2 
\nonumber \\
-2k\sum_{i=1}^n \sigma_i a(q_{r}^i(\bxi), \bxi_i)-2 k\sum_{i=1}^n E^i(\bu_h)(\sigma_i\bxi_i) 
-2k\sum_{i=1}^n \ve_{a}^i(\bu_h) (\sigma_i \bxi_i).
\end{align}
As earlier, using (\ref{sumbp}), we note that
\begin{align}\label{l2eebxi02}
2k\sum_{i=1}^n \sigma_i a(q_{r}^i(\bxi), \bxi_i)= 2k\sum_{i=1}^n \gamma a(\hbxi_i, 
\sigma_i\bxi_i)-2k\sum_{i=2}^n k\sum_{j=1}^{i-1} \pt \beta(t_i-t_j) a(\hbxi_j, 
\sigma_i\bxi_i).
\end{align}
The first term can be handled as follows (for some $\ve >0$):
\begin{align}\label{l2eebxi02a}
2k\sum_{i=1}^n \gamma a(\hbxi_i, \sigma_i\bxi_i) \le \ve k\sum_{i=1}^n \sigma_i 
\|\nabla\bxi_i\|^2+K(\ve, \mu,\gamma) k\sum_{i=1}^n e^{2\alpha t_i} 
\|\nabla\hbxi_i\|^2.
\end{align}
For the second term, using similar technique as in (\ref{pree09}), we observe that
\begin{align}\label{l2eebxi02b}
& 2k\sum_{i=2}^n k\sum_{j=1}^{i-1} \pt \beta(t_i-t_j) a(\hbxi_j, \sigma_i\bxi_i)
\le \ve k\sum_{i=1}^n \sigma_i \|\nabla\bxi_i\|^2 \\
+Kk\sum_{i=2}^n &\Big(k\sum_{j=1}^{i-1} e^{-\delta(t_i-t_j)} \big(\frac{e^{\delta k-1}}{k}\big) e^{\alpha t_i}\|\nabla\hbxi_j\|\Big)^2
\le \ve k\sum_{i=1}^n \sigma_i \|\nabla\bxi_i\|^2+Kk\sum_{i=1}^n e^{2\alpha t_i}\|\nabla\hbxi_j\|^2. \nonumber
\end{align}
Combining (\ref{l2eebxi02})-(\ref{l2eebxi02b}), we find that
\begin{align}\label{l2eebxi2}
2k\sum_{i=1}^n \sigma_i a(q_{r}^i(\bxi), \bxi_i) \le \ve k\sum_{i=1}^n \sigma_i 
\|\nabla\bxi_i\|^2+K k\sum_{i=1}^n e^{2\alpha t_i}\|\nabla\hbxi_i\|^2.
\end{align}
From Lemma \ref{e-ve}, we obtain for $r=0$ and $s=1$
\begin{equation}\label{l2eebxi4}
2k\sum_{i=1}^n \big\{E^i(\bu_h)(\sigma_i\bxi_i)+\ve_a^i(\bu_h)(\sigma_i\bxi_i)\big\}
\le \ve k\sum_{i=1}^n \sigma_i \|\nabla\bxi_i\|^2+Kk^2 (1+\log\frac{1}{k}) e^{2\alpha t_n}.
\end{equation}
Incorporate the estimates (\ref{l2eebxi2})-(\ref{l2eebxi4}) in 
(\ref{l2eebxi01}) and choose $\ve=\mu/2$ to conclude
\begin{align*}
\sigma_n\|\bxi_n\|^2 + (\mu-\frac{e^{2\alpha k}-1}{k\lambda_1}) k\sum_{i=1}^n 
\sigma_i\|\nabla\bxi_i\|^2\le Kk^2 (1+\log\frac{1}{k}) e^{2\alpha t_n}+
K k\sum_{i=1}^n e^{2\alpha t_i}\|\nabla\hbxi_i\|^2.
\end{align*}
We multiply by $e^{2\alpha t_i}$ and use Lemma \ref{ieens} to complete the rest of the proof.
\end{proof}
\noindent
We now obtain estimates of $\bta$ below. Hence forward,  $K_T$ means $KT e^{KT}.$
\begin{lemma}\label{estbta}
Assume (${\bf A1}$), (${\bf A2}$) and a spatial discretization scheme that satisfies
conditions (${\bf B1}$), (${\bf B2}$) and (${\bf B4}$). Further, assume that $\bU^n$ and 
$\bV^n$ satisfy (\ref{fdbej}) and (\ref{fdbejv}), respectively. Then, for some positive constant $K,$ there holds
\begin{eqnarray}
\|\bta_n\|^2+e^{-2\alpha t_n} k\sum_{i=1}^n e^{2\alpha t_i}\|\bta_i\|^2 \le K_{t_n}k(1+\log\frac{1}{k}), \label{estbta1} \\
\|\bta_n\|_1^2+e^{-2\alpha t_n} k\sum_{i=1}^n e^{2\alpha t_i}\|\bta_i\|_1^2 \le K_{t_n}. \label{estbta2}
\end{eqnarray}
\end{lemma}

\begin{proof}
We shall only prove the first estimate as the second one will follow similarly.
For $n=i,$ we put $\bphi_h=\bta_i$ in (\ref{eebenl}), multiply by $ke^{2\alpha t_i}$
and sum over $1\le i\le n\le N$ to obtain as in (\ref{pree02})
\begin{align}\label{estbta01}
\|\tbta_n\|^2 +2\mu k\sum_{i=1}^n \|\nabla\tbta_i\|^2\le k\sum_{i=1}^{n-1} \frac{(e^{2\alpha k}-1)}{k}\|\tbta_i\|^2+2k\sum_{i=1}^n e^{2\alpha t_i} \Lambda_h^i(\bta_i).
\end{align}
We recall from (\ref{dLbe}) that
\begin{align}\label{estbta02}
\Lambda^i_h(\bta_i)= -b(\bu_h^i,\bxi_i,\bta_i)-b(\e_i,\bu_h^i,\bta_i)
-b(\e_i,\bxi_i,\bta_i).
\end{align}
Using (\ref{nonlin1}) and Lemma \ref{pree}, we obtain the following estimates:
\begin{eqnarray}
b(\e_i,\bxi_i,\bta_i)& \le & \|\bxi_i\|^{1/2}\|\nabla\bxi_i\|^{3/2}\|\nabla\bta_i\| +\|\nabla\bxi_i\|\|\bta_i\|\|\nabla\bta_i\| \nonumber \\
& \le & \ve \|\nabla\bta_i\|^2+K\|\bta_i\|^2+Kk^{1/2}(1+\log\frac{1}{k})^{1/2}
\|\nabla\bxi_i\|^2 \label{estbta02a} \\
b(\bu_h^i,\bxi_i,\bta_i) & \le & \ve \|\nabla\bta_i\|^2+K\|\nabla\bxi_i\|^2
 \label{estbta02b} \\
b(\e_i,\bu_h^i,\bta_i) & \le & \ve \|\nabla\bta_i\|^2+K\big(\|\nabla\bxi_i\|^2
+\|\bta_i\|^2\big). \label{estbta02c}
\end{eqnarray}
Incorporate (\ref{estbta02a})-(\ref{estbta02c}) in (\ref{estbta02}) and then in
(\ref{estbta01}). Choose $\ve= \mu/6$ and once again use Lemma \ref{pree}. Finally, use discrete Gronwall's Lemma to complete the rest of the proof.
\end{proof}

\begin{remark}
Combining Lemmas \ref{pree} and \ref{estbta}, we note that
\begin{eqnarray}
\|\e_n\|^2+e^{-2\alpha t_n} k\sum_{i=1}^n e^{2\alpha t_i}\|\e_i\|^2 \le K_{t_n}k(1+\log\frac{1}{k}), \label{err1} \\
\|\e_n\|_1^2+e^{-2\alpha t_n} k\sum_{i=1}^n e^{2\alpha t_i}\|\e_i\|_1^2 \le K_{t_n}. \label{err2}
\end{eqnarray}
Therefore, we obtain suboptimal order of convergence for $\|\e_n\|$.
\end{remark}

\noindent
Below, we shall prove optimal estimate of $\|\e_n\|$ with help of a series of Lemmas.
\begin{lemma}\label{negbta} 
Under the assumptions of Lemma \ref{estbta}, the following holds:
\begin{align}\label{negbta1}
\|\bta_n\|_{-1}^2+e^{-2\alpha t_n} k\sum_{i=1}^n e^{2\alpha t_i} \|\bta_i\|^2
\le K_{t_n}k^2.
\end{align}
\end{lemma}

\begin{proof}
Put $\bphi_h= e^{2\alpha t_i} (-\td_h)^{-1}\bta_i$ in (\ref{eebenl}) for $n=i$. Multiply the 
equation by $ke^{2\alpha ik}$ and sum over $1\le i\le n\le N$ to arrive at
\begin{align}\label{negbta01}
\|\tbta_n\|_{-1}^2+2\mu k &\sum_{i=1}^n  \|\tbta_i\|^2 \le \sum_{i=1}^{n-1} 
(e^{2\alpha k}-1)\|\tbta_i\|_{-1}^2+2k \sum_{i=1}^n e^{2\alpha t_i}
\Lambda_h^i (e^{2\alpha t_i} (-\td_h)^{-1}\bta_i ).
\end{align}
From (\ref{dLbe}), we find that
\begin{align}\label{negbta02}
 |2\Lambda_h^i((-\td_h)^{-1}\bta_i)| & \le |2b(\e_i,\bu_h^i,(-\td_h)^{-1}\bta_i) 
\nonumber \\
& +b(\bu_h^i,\e_i,(-\td_h)^{-1}\bta_i) -b(\e_i,\e_i,(-\td_h)^{-1}\bta_i)|.
\end{align}
For the first term on the right hand-side of (\ref{negbta02}), we use (\ref{nonlin1}) to find 
that
\begin{equation}\label{negbta03}
|2b(\e_i,\bu_h^i,-\td_h^{-1}\bta_i)| \le K\|\e_i\|\|\bu_h^i\|_1\|\bta_i\|_{-1}^{1/2}
\|\bta_i\|^{1/2}.
\end{equation}
Also,
\begin{align}\label{negbta04}
|2b(\bu_h^i,\e_i,-\td_h^{-1}\bta_i)|& \le |(\bu_h^i\cdot\nabla\e_i, -\td_h^{-1}\bta_i)|+
|(\bu_h^i\cdot\nabla(-\td_h^{-1}\bta_i),\e_i)| \nonumber \\
& \le |(\bu_h^i\cdot\nabla\e_i, -\td_h^{-1}\bta_i)|+K\|\bu_h^i\|_1 
\|\bta_i\|_{-1}^{1/2}\|\bta_i\|^{1/2}\|\e_i\|.
\end{align}
For $D_1=\frac{\partial}{\partial x_1}$, we note that
\begin{align}\label{negbta05}
(\bu_h^i\cdot\nabla\e_i,-\td_h^{-1}\bta_i) &=\sum_{l,j=1}^2 \int_{\Omega}\bu_{h,l}^i
D_l(\e_{i,j})(-\td_h^{-1})\bta_{i,j}dx \nonumber \\
&= -\sum_{l,j=1}^2 \int_{\Omega}\big\{D_l(\bu_{h,l}^i)\e_{i,j}(-\td_h^{-1})\bta_{i,j}+
\bu_{h,l}^i\e_{i,j}D_l((-\td_h^{-1})\bta_{i,j})\}dx. \nonumber \\
& \le K\|\bu_h^i\|_1\|\e_i\|\|\bta_i\|_{-1}^{1/2}\|\bta_i\|^{1/2}.
\end{align}
Finally, from (\ref{nonlin1}), we find that
\begin{align}\label{negbta06}
|2b(\e_i,\e_i,-\td_h^{-1}\e_i)| \le K\|\e_i\|\big(\|\e_i\|_1+\|\e_i\|^{1/2}
\|\e_i\|_1^{1/2}\big)\|\bta_i\|_{-1}^{1/2}\|\bta_i\|^{1/2}.
\end{align}
Now, combine (\ref{negbta02})-(\ref{negbta06}) and use the fact that
$$ \|\e_i\|_1 \le \|\bu_h^i\|_1+\|\bU^i\|_1 \le K$$
to observe that
\begin{align}\label{negbta07}
 |2\Lambda_h^i((-\td_h)^{-1}\bta_i)| & \le K\|\e_i\|\|\bta_i\|_{-1}^{1/2} 
\|\bta_i\|^{1/2} \nonumber \\
 & \le K\|\bxi_i\|\|\bta_i\|_{-1}^{1/2} \|\bta_i\|^{1/2}
 +K\|\bta_i\|_{-1}^{1/2}\|\bta_i\|^{3/2}.
\end{align}
Incorporate (\ref{negbta07}) in (\ref{negbta01}) and use kickback argument to obtain
\begin{align}\label{negbta08}
\|\tbta_n\|_{-1}^2+\mu k &\sum_{i=1}^n  \|\tbta_i\|^2 \le Kk\sum_{i=1}^n 
\|\tbta_i\|_{-1}^2+Kk \sum_{i=1}^n \|\tbxi\|^2.
\end{align}
Finally, use Lemma \ref{neg}, apply discrete Gronwall's lemma and multiply the resulting 
estimate by $e^{-2\alpha t_i}$ to complete the rest of the proof.
\end{proof}

\begin{remark}
From Lemmas \ref{neg} and \ref{negbta}, we have the following estimate
\begin{equation}\label{l2ee}
e^{-2\alpha t_n} k\sum_{i=1}^n e^{2\alpha t_i}\|\e_i\|^2 \le K_{t_n}k^2.
\end{equation}
\end{remark}

We need another estimate of $\bta$ similar to the one in Lemma \ref{ieens} and the proof 
will follow in a similar line. For that purpose, we multiply (\ref{eebenl}) by $k$ and sum over $1$ to $n$ and similar to (\ref{eebelin1i}), we obtain
\begin{align}\label{eebenli}
 (\pt\hbta_n,\bphi_h)+ \mu a(\hbta_n,\bphi_h)+ a(q^n_r(\hbta),\bphi_h) = k\sum_{i=1}^n 
\Lambda_h^i(\bphi_h).
\end{align}

\begin{lemma}\label{ibta}
Under the assumptions of Lemma \ref{estbta}, the following holds:
\begin{align}\label{ibta1}
\|\hbta_n\|^2+ e^{-2\alpha t_n} k\sum_{i=1}^n e^{2\alpha t_i} \|\nabla\hbta_i\|^2 \le  K_{t_n}k^2(1+\log\frac{1}{k}).
\end{align}
\end{lemma}

\begin{proof}
Choose $\bphi_h=\hbta_i$ in (\ref{eebenli}) for $n=i$, multiply by $k e^{2\alpha t_i}$ and 
then sum over $1\le i\le n$ to observe as in (\ref{ieens01})
\begin{align}\label{ibta01}
e^{2\alpha t_n}\|\hbta_n\|^2+\mu k\sum_{i=1}^n e^{2\alpha t_i}\|\nabla\hbta_i\|^2 \le 
k\sum_{i=1}^n e^{2\alpha t_i} k\sum_{j=1}^i |\Lambda_h^i(\hbta_i)|.
\end{align}
We observe that
\begin{equation}\label{ibta02}
k\sum_{j=1}^i |\Lambda_h^j(\hbta_i)|= k\sum_{j=1}^i \Big|b(\bu_h^j, \e_j,\hbta_i)+b(\e_j,\bu_h^j,\hbta_i)+b(\e_j,\e_j,\hbta_i)\Big|.
\end{equation}
Use (\ref{nonlin1}), (\ref{err1}) and (\ref{err2}) to obtain
\begin{align}\label{ibta02a}
& k\sum_{i=1}^n e^{2\alpha t_i}k\sum_{i=1}^n |b(\e_j,\e_j,\hbta_i)|\le 
Kk\sum_{i=1}^n e^{2\alpha t_i}k\sum_{j=1}^i\|\e_j\|^{1/2} \|\e_j\|^{3/2}_1
\|\nabla\hbta_i\| \nonumber \\
\le & Kk\sum_{i=1}^n e^{2\alpha t_i}\Big(k\sum_{j=1}^i\|\tilde{\e}_j\|^2\Big)^{1/4} \Big(k\sum_{j=1}^i\|\tilde{\e}_j\|_1^2\Big)^{3/4} \|\nabla\hbta_i\|  \nonumber \\
\le & K_{t_n}k^2(1+\log\frac{1}{k})+\ve k\sum_{i=1}^n e^{2\alpha t_i}\|\nabla\hbta_i\|^2.
\end{align}
Similarly,
\begin{align}\label{ibta02b}
k\sum_{i=1}^n e^{2\alpha t_i}k\sum_{j=1}^i |b(\e_j,\bu_h^j,\hbta_i)|\le 
& Kk\sum_{i=1}^n e^{2\alpha t_i}\Big(k\sum_{j=1}^i\|\tilde{\e}_i\|^2\Big)^{1/2}
\|\nabla\hbta_i\| \nonumber \\
 \le & K_{t_n}k^2+\ve k\sum_{i=1}^n e^{2\alpha t_i}\|\nabla\hbta_i\|^2.
\end{align}
and
\begin{align}\label{ibta02c}
k\sum_{i=1}^n e^{2\alpha t_i} k\sum_{j=1}^i |b(\bu_h^i,\e_i,\bphi_h)| \le k
\sum_{i=1}^n e^{2\alpha t_i} k\sum_{j=1}^i \Big(\frac{1}{2}|((\nabla\cdot\bu_h^i) \e_i,\bphi_h)|+|((\bu_h^i\cdot\nabla)\bphi_h,\e_i)|\Big) \nonumber \\
\le Kk\sum_{i=1}^n e^{2\alpha t_i}\Big(k\sum_{j=1}^i\|\tilde{\e}_i\|^2\Big)^{1/2}
\|\nabla\hbta_i\|\le K_{t_n}k^2+\ve k\sum_{i=1}^n e^{2\alpha t_i}\|\nabla\hbta_i\|^2.
\end{align}
Combining these estimates, namely; (\ref{ibta02a})-(\ref{ibta02c}) and putting
 $\ve=\mu/6$, we conclude the rest of the proof.
\end{proof}
\noindent We present below a Lemma with optimal estimate for $\bta_n$.
\begin{lemma}\label{eebta} 
Under the assumptions of Lemma \ref{estbta}, the following holds:
\begin{align}\label{eebta1}
t_n\|\bta_n\|^2+e^{-2\alpha t_n} k\sum_{i=1}^n \sigma_i\|\bta_i\|_1^2 \le K_{t_n}k^2(1+\log\frac{1}{k}).
\end{align}
\end{lemma}

\begin{proof}
We choose $\bphi_h=\sigma_i\bta_i$ in (\ref{eebenl}) for $n=i$. Multiply the 
resulting equation by $k$ and sum it over $1<i<n$ to find that
\begin{align}\label{eebta01}
\sigma_n\|\bta_n\|^2+2\mu k\sum_{i=1}^n \sigma_i\|\nabla\bta_i\|^2 & \le K(\alpha) k\sum_{i=2}^{n-1}\|\tbta_i\|^2-2k\sum_{i=1}^n a(q_{r}^i(\bta),\sigma_i\bta_i) \nonumber \\
&+2k\sum_{i=1}^n\Lambda_h^i(\sigma_i\bta_i).
\end{align}
As in (\ref{l2eebxi2}), we obtain
\begin{align}\label{eebta02}
2k\sum_{i=1}^n \sigma_i a(q_{r}^i(\bta), \bta_i) \le \ve k\sum_{i=1}^n \sigma_i 
\|\nabla\bta_i\|^2+K k\sum_{i=1}^n e^{2\alpha t_i}\|\nabla\hbta_i\|^2.
\end{align}
Following the proof technique leading to the estimate (\ref{estbta02}), we observe that
\begin{align}\label{eebta03}
2k\sum_{i=1}^n\Lambda_h^i(\sigma_i\bxi_i) \le \ve k\sum_{i=1}^n \sigma_i\|\nabla\bta_i\|^2
+Kk\sum_{i=1}^n\sigma_i \big(\|\nabla\bxi_i\|^2+\|\bta_i\|^2\big).
\end{align}
Substitute (\ref{eebta02})-(\ref{eebta03})  in  (\ref{eebta01})  and this completes the rest of the proof.
\end{proof}

\begin{theorem}\label{l2eebe}
 Under the assumptions of Lemma \ref{estbta}, following holds:
\begin{equation}\label{l2eebe1}
 \|\e_n\| \le K_Tt_n^{-1/2}k(1+\log \frac{1}{k})^{1/2}.
\end{equation}
\end{theorem}

\begin{proof}
Combine the Lemmas \ref{l2eebxi} and \ref{eebta} to complete the rest of the proof.
\end{proof}

\begin{remark}
We need not split the error $\e$ in $\bxi$ and $\bta$ in order to obtain optimal error estimate (\ref{l2eebe1}).
However for optimal error estimate in $L^2$-norm which is uniform in time, we need to split the 
error  $\e_n = \bta_n-\bxi_n.$ 
\end{remark}

\section[Uniform Error]{Uniform Error Estimate}
\se

In this section, we prove the estimate (\ref{l2eebe1}) to be uniform under the uniqueness 
condition $\mu-2N\nu^{-1}\|\f\|_{\infty} >0,$ where $N$ is given as in (\ref{uc}).
We observe that the estimate (\ref{l2eebxi1})
involving $\bxi_n$  is uniform in nature. Hence, we are left to deal with $\bL^2$
estimate of $\bta_n.$

\begin{lemma}\label{btaunif}
Let the assumptions of Lemma \ref{estbta} hold. Under the uniqueness condition 
$\mu-2N\nu^{-1}\|\f\|_{\infty} >0$ and under the assumption
$$ 1+(\frac{\mu\lambda_1}{2})k > e^{2\alpha k}, $$
which holds for $0<k<k_0,~k_0>0$, the following uniform estimate hold:
\begin{equation}\label{btaunif1}
 \|\bta_n\| \le K\tau_n^{-1/2}k(1+\log \frac{1}{k})^{1/2},
\end{equation}
where $\tau_n=\min \{1,t_n\}$.
\end{lemma}

\begin{proof}
Choose $\bphi_h=\bta_i$ in (\ref{eebenl}) for $n=i$ to obtain
\begin{align}\label{001}
\pt\|\bta_i\|^2+2\mu \|\nabla\bta_i\|+2 a(q^i_r(\bta),\bta_i) \le 2\Lambda^i_h(\bta_i).
\end{align}
From (\ref{dLbe}), we find that
\begin{align}\label{002}
\Lambda^i_h(\bta_i)= -b(\bxi_i,\bu^i_h,\bta_i)-b(\bu_h^i,\bxi_i,\bta_i)
-b(\bta_i,\bu^i_h,\bta_i)-b(\e_i,\bxi_i,\bta_i).
\end{align}
From the definition of $N$ (see (\ref{uc})), we note that
\begin{align}\label{004}
|b(\bta_i,\bu^i_h,\bta_i)| \le N\|\nabla\bta_i\|^2\|\nabla\bu^i_h\|.
\end{align}
Again with the help of (\ref{nonlin1}), we obtain
\begin{align}\label{005}
|b(\bxi_i,\bu^i_h,\bta_i)|&+|b(\bu_h^i, \bxi_i,\bta_i)|\le K\|\bxi\|\|\nabla 
\bta_i\|\|\nabla\bu^i_h\|^{1/2} \|\td_h\bu^i_h\|^{1/2}  \nonumber \\ 
& \le K\tau_i^{-3/4}k(1+\log\frac{1}{k})^{1/2}\|\nabla\bta_i\|.
\end{align}
Since $\|\e_i\|_2 \le \|\bu_h^i\|_2+\|\bU^i\|_2 \le
Kt_i^{-1/2}$, we conclude that
\begin{equation}\label{005a}
|b(\e_i,\bxi_i,\bta_i)| \le K\tau_i^{-3/4}k(1+\log\frac{1}{k})^{1/2}\|\nabla\bta_i\|.
\end{equation}
Therefore, from (\ref{004})-(\ref{005a}), we find that
\begin{align}\label{006}
|\Lambda^i_h(\bta_i)| \le N\|\nabla\bta_i\|^2\|\nabla\bu^i_h\|+K\tau_i^{-3/4}k 
(1+\log\frac{1}{k})^{1/2}\|\nabla\bta_i\|.
\end{align}
We recall from \cite{WHS} that
$$ \limsup_{t\to\infty} \|\nabla\bu_h(t)\| \le \nu^{-1}\|\f\|_{\infty}, $$
and therefore, for large enough $i\in \mathbb{N}$, say $i>i_0$ we obtain from (\ref{006})
\begin{align}\label{007}
 |2\Lambda_h^i(\e_i)| \le 2N\nu^{-1}\|\f\|_{\infty}\|\nabla\bta_i\|^2+K\tau_i^{-3/4} 
k (1+\log\frac{1}{k})^{1/2}\|\nabla\bta_i\|.
\end{align}
With $\sigma_i=\tau_i e^{2\alpha t_i},$ we multiply (\ref{001}) by $k\sigma_i$ and sum 
over $i_0+1$ to $n$ to obtain
\begin{align}\label{008}
k\sum_{i=i_0+1}^n e^{2\alpha t_i} \{\pt\|\bta_i\|^2+2\mu \|\nabla\bta_i\|^2\}
+2k\sum_{i=i_0+1}^n e^{2\alpha t_i} a(q^i_r(\bta),\bta_i) \nonumber \\
\le 2k\sum_{i=i_0+1}^n \sigma_i \Lambda^i_h(\bta_i).
\end{align}
Without loss of generality, we can assume that $i_0$ is big enough, so that, by definition
$\tau_i=1$ for $i\ge i_0$. We rewrite (\ref{008}) as follows:
\begin{align}\label{008a}
k\sum_{i=i_0+1}^n e^{2\alpha t_i} \{\pt\|\bta_i\|^2+2\mu \|\nabla\bta_i\|^2\}
+2k\sum_{i=1}^n e^{2\alpha t_i} a(q^i_r(\bta),\bta_i) \nonumber \\
\le 2k\sum_{i=i_0+1}^n \sigma_i \Lambda^i_h(\bta_i)+2k\sum_{i=1}^{i_0} e^{2\alpha t_i} a(q^i_r(\bta),\bta_i).
\end{align}
We observe that the last term on the left hand-side of (\ref{008a}) is non-negative and hence is dropped.
\begin{align}\label{009}
e^{2\alpha t_n}\|\bta_n\|^2-\sum_{i=i_0+1}^{n-1} e^{2\alpha t_i} (e^{2\alpha k}-1)\|\bta_i\|^2+\mu k\sum_{i=i_0+1}^n e^{2\alpha i}  \|\nabla\bta_i\|^2 \nonumber \\
+ k\sum_{i=i_0+1}^n (\mu-2N\nu^{-1}\|\f\|_{\infty}) e^{2\alpha i}  \|\nabla\bta_i\|^2 
\le e^{2\alpha t_{i_0}}\|\bta_{i_0}\|^2+2k\sum_{i=1}^{i_0} e^{2\alpha t_i} q^i_r(\|\nabla\bta_i\|)\|\nabla\bta_i\| \nonumber \\
+Kk^2\sum_{i=i_0+1}^n \tau_i^{1/4} e^{2\alpha t_i} (1+\log\frac{1}{k})^{3/4}\|\nabla\bta_i\|.
\end{align}
Under the assumption
$$ 1+(\frac{\mu\lambda_1}{2})k > e^{2\alpha k}, $$
which holds for $0<k<k_0,~k_0>0$ with $0 \le \alpha \le \min \big\{\delta, 
\frac{\mu\lambda_1}{2}\big\},$ we find that
\begin{align}\label{010}
\frac{\mu}{2} k\sum_{i=i_0+1}^n e^{2\alpha t_i}  \|\nabla\bta_i\|^2-\sum_{i=i_0+1}^{n-1} 
e^{2\alpha t_i} (e^{2\alpha k}-1)\|\bta_i\|^2 \nonumber \\
=k\sum_{i=i_0+1}^{n} \big(\frac{\mu}{2}-\frac{e^{2\alpha k}-1}{k\lambda_1}\big)\sigma_i
\|\nabla\bta_i\|^2 \ge 0.
\end{align}
Due to uniqueness condition, we arrive at the following:
\begin{align}\label{011}
 k\sum_{i=i_0+1}^n (\mu-2N\nu^{-1}\|\f\|_{\infty}) e^{2\alpha i}  \|\nabla\bta_i\|^2 
\ge 0.
\end{align}
Following the proof techniques of (\ref{pree08})-(\ref{pree09}), we obtain
\begin{align}\label{012}
2k\sum_{i=1}^{i_0} e^{2\alpha t_i} q^i_r(\|\nabla\bta_i\|) \|\nabla\bta_i\| \le K 
k\sum_{i=1}^{i_0} e^{2\alpha t_i} \|\nabla\bta_i\|^2.
\end{align}
And
\begin{align}\label{013}
Kk^2\sum_{i=i_0+1}^n \tau_i^{1/4} e^{2\alpha t_i} (1+\log\frac{1}{k})^{1/2} 
\|\nabla\bta_i\|\le \frac{\mu}{4} k\sum_{i=i_0+1}^n \sigma_i  \|\nabla\bta_i\|^2
\nonumber \\
+Kk^2 (1+\log\frac{1}{k}) k\sum_{i=i_0+1}^n e^{2\alpha t_i} \tau_i^{-1/2}.
\end{align}
Incorporate (\ref{010})-(\ref{013}) in (\ref{009}), use Lemma \ref{negbta} and
(\ref{l2ee}) to observe that
\begin{align}\label{014}
e^{2\alpha t_n}\|\bta_n\|^2+k\sum_{i=1}^n \sigma_i  \|\nabla\bta_i\|^2 & \le K_{t_0} 
k^2+Kk^2 (1+\log\frac{1}{k}) e^{2\alpha t_n} \nonumber \\
&+Kk\sum_{i=1}^{i_0} e^{2\alpha t_i} \|\nabla\bta_i\|^2.
\end{align}
Multiply by $e^{-2\alpha t_i}$ and under the assumption that
\begin{equation}\label{015}
k\sum_{i=1}^{t_0} e^{2\alpha t_i} \|\nabla\bta_i\|^2 \le
K_{t_0}t_{t_0}^{-1}k^2(1+\log \frac{1}{k}).
\end{equation}
we conclude that
$$ \|\bta_n\| \le Kt_n^{-1/2}k(1+\log\frac{1}{k})^{1/2}, $$
since $i_0>0$ is fixed.
Combining this result with (\ref{l2eebxi1}) we complete the rest of the proof.
\end{proof}
\noindent We are now left with the proof (\ref{015}).
\begin{lemma}\label{bta1}
Under the assumption of Lemma \ref{estbta}, the following holds
$$ k\sum_{i=1}^{i_0} e^{2\alpha t_i} \|\nabla\bta_i\|^2 \le 
K_{t_{i_0}} t_{i_0}^{-1}k^2(1+\log \frac{1}{k}). $$
\end{lemma}

\begin{proof}
In (\ref{estbta01}), we use
\begin{align*}
\Lambda^i_h(\bta_i) &= -b_h(\bu^i_h,\e_i,\bta_i)-b_h(\e_i,\bU^i,\bta_i) \nonumber \\
& \le \frac{\mu}{4}\|\nabla\bta_i\|^2+K\|\e_i\|^2(\|\td_h\bu^i_h\|+\|\td_h\bU^i\|),
\end{align*}
along with Lemma \ref{negbta} and Theorem \ref{l2eebe} to arrive at 
\begin{align}\label{bta02}
\|\tbta_{i_0}\|^2 +\mu k\sum_{i=1}^{i_0} \|\nabla\tbta_i\|^2\le K_{t_{i_0}}k^2 +K_{t_{i_0}}t_{i_0}^{-1}k^2(1+\log \frac{1}{k}) k\sum_{i=1}^{i_0} e^{2\alpha t_i} t_i^{-1/2}.
\end{align}
This completes the rest of the proof.
\end{proof}
\section{Conclusion}
\se
In this paper, we have discussed optimal error estimates for the backward Euler
method employed to the Oldroyd model with non-smooth initial data, the is,
$\bu_0\in \bJ_1$. We have proved both optimal and uniform error estimate for
the velocity. Uniform estimate is proved under uniqueness condition. The error
analysis for the non-smooth initial data tells us that we need a few more proof
techniques than the smooth case and proofs are more involved.



\begin{thebibliography}{20}

\bibitem{AS89}
Agranovich, Yu. Ya. and Sobolevskii, P. E. ,
{\it Investigation of viscoelastic fluid mathematical model},
RAC. Ukranian SSR. Ser. A, 10 (1989), 71-74.

\bibitem{AO}
Akhmatov, M. M. and Oskolkov, A. P. ,
{\it On convergent difference schemes for the equations of motion of an Oldroyd fluid},
J. Soviet Math., 47 (1989), 2926-2933.

\bibitem{BF}
Brezzi, F. and Fortin, M. , 
{\it Mixed and Hybrid finite element methods},
Springer-Verlag, New York, 1991.

\bibitem{BP}
Bercovier, M. and Pironneau, O. ,
{\it Error estimates for finite element solution of the Stokes problem in the primitive 
variables},
Numer. Math., 33 (1979), 211-224.

\bibitem{EO92}
Emel'yanova, D. V. and Oskolkov, A. P. ,
{\it Certain nonlocal problems for two-dimensional equations of motion of Oldroyd fluids}, 
Zap. Nauchn. Sem. Leningrad. Otdel. Mat. Inst. Steklov. (LOMI) 189 (1991), 101-121;
English translation in J. Soviet Math. 62 (1992), 3004-3016.

\bibitem{GR}
Girault, V. and Raviart, P. A. ,
{\it Finite Element Approximation of the Navier-Stokes Equations},
Lecture Notes in Mathematics No. 749, Springer , New York, 1980.

\bibitem{G11}
Goswami, D., {\it Finite Element Approximation to the Equations of Motion Arising in Oldroyd Viscoelastic Model of Order One},
Ph.D. Dissertation, Department of Mathematics, IITBombay (2011).

\bibitem{GP11}
Goswami, D. and Pani, A. K.,
{\it A Priori Error Estimates for Semidiscrete Finite Element Approximations to the 
Equations of Motion Arising in Oldroyd Fluids of Order One},
Int. J. Numer. Anal. Model., vol 8, no 2 (2011), 324-352.

\bibitem{HLST}
He, Y. , Lin, Y. , Shen, S. S. P. and Tait, R. ,
{\it On the convergence of Viscoelastic fluid flows to a steady state},
Adv. Differential Equations 7 (2002), 717-742.

\bibitem{HLSST}
He, Y. , Lin, Y. , Shen, S.S.P. , Sun, W. and Tait, R. ,
{\it Finite element approximation for the viscoelastic fluid motion problem},
J. Comp. Appl. Math. 155 (2003), 201-222.

\bibitem{HR82}
Heywood, J. G. and Rannacher, R. ,
{\it Finite element approximation of the nonstationary Navier-Stokes problem: I. 
Regularity 
of solutions and second order error estimates for spatial discretization},
SIAM J. Numer. Anal. 19 (1982), 275-311.

\bibitem{KrKO91}
Karzeeva, N. A., Kotsiolis, A. A. and Oskolkov, A. P. , 
{\it On dynamical system generated by initial-boundary value problems for the equations 
of motion of linear viscoelastic  fluids},
(Russian) Trudy Mat. Inst. Steklov. 188 (1990), 59-87; 
English translation in Proc. Steklov Inst. Math. (1991), no. 3, 73-108.
\bibitem{KOS92}
Kotsiolis, A. A. , Oskolkov, A. P. and Shadiev, R. D. , {\it A priori 
estimate on the semiaxis $t \ge 0$ for the solutions of the equations of 
motion of linear viscoelastic fluids with an infinite Dirichlet integral, 
and their applications}, (Russian. English
summary) Zap. Nauchn. Sem. Leningrad. Otdel. Mat. Inst. Steklov. (
LOMI) 182 (1990),
86-101; English translation in J. Soviet Math. 62 (1992), no. 3, 2777-2788.
\bibitem{MT}
McLean, W. and Thom\'{e}e, V., {\it Numerical solution of
an evolution equation with a positive type memory term},
J. Austral. Math. Soc. Ser. B 35 (1993), 23-70.
\bibitem{Old}
Oldroyd, J. G. , {\it Non-Newtonian flow of liquids and solids. Rheology: Theory
and Applications}, Vol. I (F. R. Eirich, Ed.), AP, New York (1956), 653-682.
\bibitem{Os89}
Oskolkov, A. P. , 
{\it Initial boundary value problems for the equations of motion of Kelvin-Voigt fluids and 
Oldroyd fluids}, 
(Russian) Trudy Mat. Inst. Steklov. 179 (1988), 126-164;
 English translated in Proc. Steklov. Inst. Math. (1989), no. 2, 137-182.
\bibitem{PS98}
Pani, A. K. and Sinha, R. K.,
{\it On the backward {E}uler method for time dependent parabolic integro-differential 
equations with nonsmooth initial data},
J. Integral Equations Appl. 10 (1998), 219-249.
\bibitem{PS198}
Pani, A. K. and Sinha, R. K.,
{\it Quadrature based finite element approximations to time dependent parabolic equations
with nonsmooth initial data},
Calcolo 35 (1998),  225-248.
\bibitem{PY05}
Pani, A. K. and Yuan, J. Y. , {\it Semidiscrete finite element Galerkin
approximations to the equations of motion arising in the Oldroyd model}, IMA J.
Numer. Anal. 25 (2005), 750-782.
\bibitem{PYD06}
Pani, A. K. , Yuan, J. Y. and Damazio, P. , {\it On a linearized backward Euler
method for the equations of motion arising in the Oldroyd fluids of order one},
SIAM J. Numer. Anal. 44 (2006), 804-825.
\bibitem{S}
Sobolevskii, P. E. , {\it Stabilization of viscoelastic fluid motion (Oldroyd's
mathematical model)}, Diff. Integ. Eq., 7(1994),  1597-1612.
\bibitem{temam}
Temam, R. , 
{\it  Navier-Stokes Equations, Theory and Numerical Analysis}, 
North-Holland , Amsterdam, 1984.
\bibitem{TZ89}
Thom\'{e}e, V. and  Zhang, N. -Y., 
{\it Error estimates for semidiscrete finite element methods for parabolic 
integro-differential equations},
 Math. Comp. 53 (1989),  121-139.
 \bibitem{WHL}
 Wang, K., He, Y. and Lin, Y.,
 {\it Long time numerical stability and asymptotic analysis for the viscoelastic Oldroyd flows},
 Disc. Cont. Dyn. Sys. Ser. B, 17(2012), 1551-1573.
\bibitem{WHS}
Wang, K. and He, Y. and Shang, Y.,
{\it Fully discrete finite element method for the viscoelastic fluid motion equations},
Disc. Cont. Dyn. Sys. Ser. B, 13 (2010), 665-684.


\end{thebibliography}
\end{document}